\documentclass{amsart}
\RequirePackage[utf8]{inputenc}
\usepackage{amsmath,amsthm,amsfonts,amssymb,amscd,amsbsy,multirow,hyperref}
\usepackage[all]{xy}
\usepackage{graphicx,pgf,caption,subcaption}
\usepackage{enumerate}
\usepackage{dsfont}
\usepackage{stmaryrd}
\usepackage{ytableau}

\newcommand{\dd}{\mathrm{d}}
\newcommand{\id}{\operatorname{Id}}
\newcommand{\vol}{\operatorname{vol}}

\newcommand{\Iso}{\operatorname{Iso}}
\newcommand{\Ric}{\operatorname{Ric}}
\newcommand{\scal}{\operatorname{scal}}

\newcommand{\R}{\mathds R}
\newcommand{\C}{\mathds C}

\newcommand{\SO}{\mathsf{SO}}
\renewcommand{\O}{\mathsf O}

\newcommand{\GL}{\mathsf{GL}}
\newcommand{\Ss}{\mathbb{S}}
\newcommand{\Spin}{\mathsf{Spin}}

\newcommand{\Sym}{\operatorname{Sym}}

\newcommand{\re}{\operatorname{Re}\,}
\newcommand{\im}{\operatorname{Im}\,}
\newcommand{\ii}{\sqrt{-1}\,}

\newcommand{\g}{\mathrm g}

\newcommand{\Gr}{\operatorname{Gr}_2}

\newcommand{\even}{{\operatorname{even}}}
\newcommand{\odd}{{\operatorname{odd}}}

\allowdisplaybreaks

\hypersetup{
    pdftoolbar=true,
    pdfmenubar=true,
    pdffitwindow=false,
    pdfstartview={FitH},
    pdftitle={Sectional curvature and Weitzenbock formulae},
    pdfauthor={Renato G. Bettiol and Ricardo A. E. Mendes},
    pdfsubject={},
    pdfkeywords={}
    pdfnewwindow=true,
    colorlinks=true, 
    linkcolor=blue,
    citecolor=blue,
    urlcolor=black,
}

\newtheorem{theorem}{Theorem}[]
\newtheorem{lemma}[theorem]{Lemma}
\newtheorem{proposition}[theorem]{Proposition}

\newtheorem{mainthm}{\sc Theorem}

\newtheorem{maincor}[mainthm]{\sc Corollary}

\theoremstyle{definition}

\theoremstyle{remark}
\newtheorem{remark}[theorem]{Remark}
\newtheorem{example}[theorem]{Example}

\title{Sectional curvature and Weitzenb\"ock formulae}

\dedicatory{Dedicated to the memory of Marcel Berger}

\author[R. G. Bettiol]{Renato G. Bettiol}
\address{City University of New York (Lehman College) \newline
\indent Department of Mathematics  \newline
\indent 250 Bedford Park Blvd W\newline
\indent Bronx, NY, 10468, USA }
\email{r.bettiol@lehman.cuny.edu}

\author[R. A. E. Mendes]{Ricardo A. E. Mendes}
\address{University of Oklahoma\newline
\indent Department of Mathematics\newline
\indent 601 Elm Ave\newline
\indent Norman, OK, 73019-3103, USA}
\email{ricardo.mendes@ou.edu}

\allowdisplaybreaks
\numberwithin{equation}{section}
\numberwithin{theorem}{section}


\date{\today}

\begin{document}
\begin{abstract}
We establish a new algebraic characterization of sectional curvature bounds $\sec\geq k$ and $\sec\leq k$ using only curvature terms in the Weitzenb\"ock formulae for symmetric $p$-tensors. By introducing a symmetric analogue of the Kulkarni--Nomizu product, we provide a simple formula for such curvature terms. We also give an application of the Bochner technique to closed $4$-manifolds with indefinite intersection form and $\sec>0$ or $\sec\geq0$, obtaining new insight into the Hopf Conjecture, without any symmetry assumptions.
\end{abstract}

\maketitle

\vspace{-0.3cm}

\section{Introduction}

In \emph{geometric} terms, the sectional curvature of a Riemannian manifold $(M,\g)$ is the most natural generalization to higher dimensions of the Gaussian curvature of a surface, given that it controls the behavior of geodesics. However, its \emph{algebraic} features render its study much more complicated. A substantial part of this complication arises from the fact that, even at a pointwise level, $\sec\colon\Gr(T_pM)\to\R$ is a nonlinear function defined on the Grassmannian of $2$-planes in $T_pM$, a quadric variety inside the unit sphere of $\wedge^2 T_pM$, defined by the Pl\"ucker relations. The main goal of this paper is to shed further light on the algebraic nature of sectional curvature by relating it to the curvature terms in Weitzenb\"ock formulae, which are linear endomorphisms and hence computationally more accessible.

Given a Riemannian $n$-manifold $(M,\g)$, consider the vector bundle $E\to M$ associated to the frame bundle of $M$ via a representation $\rho\colon\O(n)\to \O(E)$ of the orthogonal group. 
Geometrically relevant Laplacians $\Delta$ on sections of $E$, as the Hodge Laplacian for $E=\wedge^p TM$ and the Lichnerowicz Laplacian for $E=\Sym^p TM$, are related to the connection Laplacian $\nabla^*\nabla$ via the \emph{Weitzenb\"ock formula}
\begin{equation*}
\Delta = \nabla^*\nabla + t\,\mathcal K(R,\rho),
\end{equation*}
where $t\in\R$ is a constant and $\mathcal K(R,\rho)$ is a linear endomorphism of $E$ determined by the curvature operator $R$ of $(M,\g)$ and the orthogonal representation $\rho$, see \eqref{eq:krho} and Section~\ref{sec:weitzenformulae} for details.
As observed by Hitchin~\cite{Hitchin15}, an algebraic curvature operator $R\colon\wedge^2\R^n\to\wedge^2\R^n$ is positive-semidefinite if and only if the endomorphism $\mathcal K(R,\rho)$ is positive-semidefinite for \emph{all} irreducible representations $\rho\colon\O(n)\to \O(E)$. 
Our first result is a similar algebraic characterization of sectional curvature bounds using traceless symmetric $p$-tensors, i.e., representations $\rho\colon\O(n)\to\O( \Sym^p_0\R^n)$.

\begin{mainthm}\label{mainthm:algebraicChar}
An algebraic curvature operator $R\colon\wedge^2\R^n\to\wedge^2\R^n$ has $\sec_R\geq k$, respectively $\sec_R\leq k$, if and only if $\mathcal K(R-k\id,\Sym^p_0 \R^n)$ is positive-semidefinite, respectively negative-semidefinite, for all $p\geq2$.
\end{mainthm}

The fact that $\sec_R\leq0$ implies negative-semidefiniteness of $\mathcal K(R,\Sym^p_0\R^n)$ for all $p\geq2$ was previously obtained in \cite{HMS15}, see also~\cite{DS10}.
 
An algebraic characterization of $\sec_R\geq0$ is arguably more relevant than the analogous characterization of $R$ being positive-semidefinite for two key reasons. First, pointwise, $R$ is a linear endomorphism, hence its positive-semidefiniteness is already characterized by Sylvester's criterion. Second, globally, closed manifolds that admit metrics with positive-semidefinite $R$ have been classified~\cite[Thm.~1.13]{wilking-survey}, see also \cite{bw}, while several questions about manifolds that admit metrics with $\sec\geq0$ remain unanswered, see \cite{wilking-survey,bible,ziller-mexico} for surveys.

In order to exploit Theorem~\ref{mainthm:algebraicChar} and gain a deeper algebraic understanding of sectional curvature bounds, it is crucial to (effectively) compute the curvature terms $\mathcal K(R,\Sym^p_0\R^n)$. Our second result offers a procedure that relies on a \emph{symmetric version} $\ovee$ of the Kulkarni--Nomizu product, which we introduce in Section~\ref{sec:KNalgebras}, and the decomposition into $\O(n)$-irreducible components of
 arbitrary elements $R\in\Sym^2(\wedge^2\R^n)$, which are called \emph{algebraic modified curvature operators}.
Namely, $R=R_{\mathcal U}+R_{\mathcal L}+R_{\mathcal W}+R_{\wedge^4}\in\Sym^2(\wedge^2\R^n)$, where
$R_\mathcal U$ is the scalar curvature part, $R_\mathcal L$ is the traceless Ricci part, $R_\mathcal W$ is the Weyl tensor part, and $R_{\wedge^4}$, which vanishes if and only if $R$ satisfies the first Bianchi identity.
For completeness and comparison, we also state how curvature terms $\mathcal K(R,\wedge^p \R^n)$ for the representations $\rho\colon\O(n)\to\O(\wedge^p\R^n)$ can be computed using the classical Kulkarni--Nomizu product $\owedge$, a result obtained by Labbi~\cite[Prop.~4.2]{Labbi15} in the \emph{unmodified} case $R_{\wedge^4}=0$.

\begin{mainthm}\label{mainthm:KRWedgeSymKN}
If $R\in\Sym^2(\wedge^2\R^n)$ is an algebraic modified curvature operator, then
\begin{equation*}
\begin{aligned}
\mathcal K\big(R,\Sym^p_0 \R^n\big)&=\left(\tfrac{n+p-2}{n(p-1)}\,\mathcal K(R_{\mathcal U},\pi)+\tfrac{n+2p-4}{n(p-1)}\,\mathcal K(R_{\mathcal L},\pi)+\mathcal K(R_{\mathcal W},\pi) \right)\ovee\tfrac{\g^{\ovee(p-2)}}{(p-2)!},\\
\mathcal K\big(R,\wedge^p \R^n\big)&=\left(\tfrac{2(n-p)}{p-1}\,R_{\mathcal U}+\tfrac{n-2p}{p-1}\,R_{\mathcal L}-2\,R_{\mathcal W}+4\,R_{\wedge^4} \right)\owedge \tfrac{\g^{\owedge(p-2)}}{(p-2)!},
\end{aligned}
\end{equation*}
for all $p\geq2$ and $2\leq p\leq n-2$ respectively, where $\pi$ is the representation $\Sym^2_0 \R^n$.
\end{mainthm}

The ranges of $p$ not covered in Theorem~\ref{mainthm:KRWedgeSymKN} correspond to trivial cases, see Remark~\ref{rem:trivialcases}. For instance, if $p=1$, then $\mathcal K(R,\wedge^1\R^n)=\mathcal K(R,\Sym^1_0\R^n)=\Ric_R$ is the Ricci tensor of $R$, see Example~\ref{ex:Ricci}; this is the context where the Weitzenb\"ock formula was used by Bochner~\cite{bochner,bochner-yano} to prove vanishing theorems for harmonic $1$-forms and Killing vector fields on closed manifolds whose Ricci curvature has a sign. 
A new application of the Bochner technique, that relies on a detailed analysis of positive-definiteness of $\mathcal K(R,\wedge^p \R^n)$, was recently found in \cite{petersen-wink}.

It was observed by Berger that positive-semidefiniteness of $\mathcal K(R,\Sym^2_0 \R^n)$ is an intermediate condition between $\Ric\geq0$ and $\sec\geq0$, see Remark~\ref{rem:Berger}. Thus, by Theorem~\ref{mainthm:algebraicChar}, positive-semidefiniteness of $\mathcal K(R,\Sym^p_0 \R^n)$ for all $1\leq p\leq q$ provides a family parametrized by $q\in\mathds N$ of interpolating curvature conditions between $\Ric\geq0$ (corresponding to $q=1$) and $\sec\geq0$ (corresponding to $q=\infty$).
The convex sets that form this family are \emph{spectrahedral shadows}, for which there are efficient algorithms to check membership of a given curvature operator, see \cite{BKM}.
Moreover, since $\owedge$-products and $\ovee$-products with the metric $\g$ preserve positive-semidefinite\-ness, Theorem~\ref{mainthm:KRWedgeSymKN} provides a sufficient condition for positive-semidefiniteness of $\mathcal K(R,\wedge^p \R^n)$ and $\mathcal K(R,\Sym^p_0 \R^n)$ in terms of the components of $R$. This is also suggestive of other potential relations between positive-semidefiniteness of the endomorphisms $\mathcal K(R,\Sym^p_0 \R^n)$ for higher values of $p\geq2$, which could lead to algebraic characterizations of sectional curvature bounds more powerful  than Theorem~\ref{mainthm:algebraicChar}.


Our third result concerns simply-connected closed $4$-manifolds with $\sec>0$, which are conjecturally diffeomorphic to $S^4$ or $\C P^2$, and hence have definite intersection form. It is an application of our extension of Labbi~\cite{Labbi15} to the modified case, that is, the formula $\mathcal K(R_{\wedge^4},\wedge^2 \R^4)=4R_{\wedge^4}$ in Theorem~\ref{mainthm:KRWedgeSymKN}, together with the Bochner technique and the  \emph{Finsler--Thorpe trick}~\cite{ThorpeJDG,Thorpe72}, see also \cite{BKM}.

\begin{mainthm}\label{mainthm:dim4}
Let $(M,\g)$ be a closed simply-connected Riemannian $4$-manifold with $\sec>0$ whose intersection form is indefinite. Then the set of points where the curvature operator is not positive-definite has at least two connected components with nonempty interior.
\end{mainthm}

The question of whether $S^2\times S^2$, which has indefinite intersection form, admits metrics with $\sec>0$ is a famous unsolved problem, known as the Hopf~Conjecture.
By Theorem~\ref{mainthm:dim4}, if any such metric $\g$ existed, then there would be a region which disconnects $(S^2\times S^2,\g)$ where its curvature operator is positive-definite. Recall that closed $n$-manifolds whose curvature operator is \emph{everywhere} positive-definite are diffeomorphic to a space form $S^n/\Gamma$, see~\cite{bw}. 

Another direct consequence of Theorem~\ref{mainthm:dim4} is that a closed simply-connected K\"ahler $4$-manifold with $\sec>0$ has definite intersection form. Indeed, the curvature operator of a K\"ahler $4$-manifold is not positive-definite at any point.
Recall that a classical result of Andreotti and Frankel \cite{frankel}, using methods from complex and algebraic geometry, shows that the only K\"ahler $4$-manifold as above is $\C P^2$.

Relaxing the condition $\sec>0$ to $\sec\geq0$ on closed simply-connected $4$-manifolds, the conjectured list of possible diffeomorphism types grows to include $S^2\times S^2$, $\C P^2\#\overline{\C P^2}$, and $\C P^2\# {\C P^2}$. Apart from $\C P^2\# {\C P^2}$, these $4$-manifolds have indefinite intersection form. In this context, we have a rigidity counterpart to Theorem~\ref{mainthm:dim4}:

\begin{mainthm}\label{mainthm:dim4rigidity}
Let $(M,\g)$ be a closed simply-connected Riemannian $4$-manifold with $\sec\ge0$ whose intersection form is indefinite. Then one of the following holds:
\begin{enumerate}[\rm (i)]
\item The set of points where the curvature operator is not positive-semidefinite has at least two connected components;
\item $(M,\g)$ is isometric to the Riemannian product $\big(S^2\times S^2,\g_1\oplus\g_2\big)$ where $(S^2,\g_1)$ and $(S^2,\g_2)$ have $\sec\geq0$.
\end{enumerate} 
\end{mainthm}

As mentioned above, closed manifolds with \emph{everywhere} positive-semidef\-inite curvature operator are classified~\cite[Thm.~1.13]{wilking-survey}.

The use of symmetries led to substantial advances towards the Hopf Conjecture, notably with~\cite{hsiang-kleiner} and \cite{grove-wilking}, and more generally towards understanding manifolds with $\sec>0$ or $\sec\geq0$ via the Grove Symmetry Program~\cite{grove-survey}.
However, as these methods reach their exhaustion, it seems very important to make new progress \emph{without relying on symmetries}, as is achieved in Theorems~\ref{mainthm:dim4} and \ref{mainthm:dim4rigidity}.

Finally, Theorems~\ref{mainthm:dim4} and \ref{mainthm:dim4rigidity} also yield the following global obstruction to the existence of \emph{curvature-homogeneous} metrics, i.e., metrics whose curvature operator is the same at every point, cf.~\cite{kowalski}; see \cite{kowalski2} and \cite[\S 1.4]{gilkey-book} for surveys.

\begin{maincor}
Let $(M,\g)$ be a closed simply-connected Riemannian $4$-manifold with indefinite intersection form. If $(M,\g)$ has $\sec>0$, then it is not curvature-homogeneous. If $(M,\g)$ has $\sec\geq0$, then it can only be curvature-homogeneous if it is isometric to the product of two round spheres of possibly different radii.
\end{maincor}

This paper is organized as follows. In Section~\ref{sec:weitzenformulae}, we recall the abstract framework for Weitzenb\"ock formulae, the definition and basic properties of its curvature term $\mathcal K(R,\rho)$, some aspects of exterior and symmetric powers (of the defining $\O(n)$-representation), and provide a general overview of the Bochner technique. Section~\ref{sec:characterization} contains the proof of Theorem~\ref{mainthm:algebraicChar}. After recalling the classical Kulkarni--Nomizu product $\owedge$ for symmetric tensors on exterior powers, we introduce the product $\ovee$ for symmetric tensors on symmetric powers in Section~\ref{sec:KNalgebras}. Furthermore, some important representation theoretic aspects of these Kulkarni--Nomizu algebras are discussed in Subsection~\ref{subsec:reptheory}. Building on this material, Theorem~\ref{mainthm:KRWedgeSymKN} is proved in Section~\ref{sec:proofThmB}. Finally, Section~\ref{sec:BochnerDim4} contains the proofs of Theorems~\ref{mainthm:dim4} and \ref{mainthm:dim4rigidity}.

\subsection*{Acknowledgements} It is our pleasure to thank
Matthew Gursky for bringing to our attention an observation of Berger about $\mathcal K(R,\Sym^2_0\R^n)$ discussed in Remark~\ref{rem:Berger},
Marco Radeschi for suggestions related to the choice of $f_\varepsilon$ in the proof of Theorem~\ref{mainthm:algebraicChar}, and Uwe Semmelmann for comments related to Proposition~\ref{prop:IntegralFormula}. We also thank 
the Max Planck Institute for Mathematics in Bonn, the Westf\"alischen Wilhelms Universit\"at M\"unster, and the Universit\"at zu K\"oln for providing excellent working conditions and support throughout the development of this project.

The first-named author was supported by the National Science Foundation
grant DMS-1904342; the second-named author was supported by the Deutsche Forschungsgemeinschaft grants DFG ME 4801/1-1 and DFG SFB TRR 191.

\subsection*{Conventions}
Throughout the paper we make frequent use of the following conventions. Every finite-dimensional vector space $V$ with an inner product is identified with the dual vector space $V^*$, and we often regard the exterior power $\wedge^p V$ and the symmetric power $\Sym^p V$ as vector subspaces of the tensor power $V^{\otimes p}$.
The canonical orthonormal basis of $\R^n$ is denoted $\{e_i\}$. The corresponding orthonormal basis of the $p$th \emph{exterior power} $\wedge^p \R^n$ is $\{e_{i_1}\wedge e_{i_2}\wedge \dots\wedge e_{i_p}, i_1<i_2<\dots <i_p\}$. 
For $p=2$, there is a canonical identification $\wedge^2 \R^n\cong \mathfrak{so}(n)$ with the Lie algebra of $\SO(n)$, given by $e_i\wedge e_j\cong E_{ij}$, where $E_{ij}$ is the matrix whose $(i,j)$ entry is $+1$, $(j,i)$ entry is $-1$, and all other entries are zero.

The $p$th \emph{symmetric power} $\Sym^p \R^n$ is spanned by $\{e_{i_1}\vee e_{i_2}\vee \dots\vee e_{i_p}, i_1\leq i_2\leq \dots \leq i_p\}$. Moreover, we identify $\Sym^p \R^n$ with the space of homogeneous polynomials of degree $p$ on $\R^n$, by means of $e_{i_1}\vee e_{i_2}\vee \dots\vee e_{i_p}\cong x_{i_1} x_{i_2} \dots x_{i_p}$, where $x_i$ denote the canonical coordinates of $\R^n$.
We also denote the product of two polynomials $\varphi\in\Sym^p\R^n$ and $\psi\in\Sym^q\R^n$ by $\varphi\vee\psi\in\Sym^{p+q}\R^n$.
The inner product of two elements $\varphi,\psi \in \Sym^p\R^n$ is defined to be $\langle \varphi,\psi\rangle =\widehat\varphi(\psi)$, where $\widehat\varphi$ is the differential operator of order $p$ \emph{dual} to $\varphi$, obtained by replacing each $x_i$ with $\frac{\partial}{\partial x_i}$. For example, $\|e_i\vee e_j\|=1$ if $i\neq j$, however $\|e_i\vee e_i\|=\sqrt2$.
The space of \emph{traceless} symmetric $p$-tensors $\Sym^p_0 \R^n$ is the subspace of $\Sym^p\R^n$ consisting of \emph{harmonic} polynomials. Its orthogonal complement is isomorphic to $\Sym^{p-2}\R^n$ via multiplication by $r^2=\sum_{i=1}^n x_i^2$. 
Given $\varphi\in\Sym^p\R^n$, we shall denote by $\varphi_0$ its orthogonal projection onto $\Sym^p_0\R^n$.
The above element $r^2\in\Sym^2\R^n$ is often called $\g$, which also stands for Riemannian metrics.

Finally, our sign convention for the curvature operator $R\colon\wedge^2 T_pM\to\wedge^2T_pM$ of a Riemannian manifold $(M,\g)$ is given by
$\langle R(X\wedge Y),Z\wedge W\rangle=\langle R(X,Y)Z,W\rangle = \langle \nabla_{[X,Y]}Z-\nabla_X\nabla_Y Z+\nabla_Y\nabla_X Z,W\rangle$, so that $\sec(X\wedge Y)=R(X\wedge Y,X\wedge Y)$ if $X$ and $Y$ are orthonormal vectors. 

\section{Weitzenb\"ock formulae}
\label{sec:weitzenformulae}

In this section, we recall the abstract framework for \emph{Weitzenb\"ock formulae} following \cite{Besse,Hitchin15}, discuss the basic algebraic properties of its curvature term, and give a quick account of the so-called \emph{Bochner technique}.

Let $(M,\g)$ be an $n$-dimensional Riemannian manifold, and let $\operatorname{Fr}(TM)$ denote the principal $\O(n)$-bundle of orthonormal frames on $(M,\g)$. The structure group $\O(n)$ is replaced accordingly if $(M,\g)$ has additional structures, e.g., with $\SO(n)$ if $(M,\g)$ has an orientation, and with $\Spin(n)$ if it has a spin structure. Every orthogonal representation $\rho\colon \O(n)\to\O(E)$ gives rise to a Riemannian vector bundle $E\to M$, associated to $\operatorname{Fr}(TM)$ via $\rho$, that we also denote by $E$ or $\rho$. Moreover, $E$ has a linear connection induced by the Levi-Civita connection of $(M,\g)$, and both connections are denoted~$\nabla$. The \emph{connection Laplacian}, or \emph{rough Laplacian}, is the differential operator $\nabla^*\nabla$ on sections of $E$, where $\nabla^*$ is the $L^2$-adjoint of $\nabla$; or, in local coordinates, $\nabla^*\nabla=-\sum_{i=1}^k \nabla_{E_i} \nabla_{E_i}$ where $\{E_i\}$ is an orthonormal frame for $TM$. Note that a section $\sigma$ of $E$ satisfies $\nabla^*\nabla \sigma=0$ if and only if $\sigma$ is parallel.

The \emph{Weitzenb\"ock formula} relates the connection Laplacian and other geometrically relevant \emph{Laplacians} on sections of $E$, collectively denoted by $\Delta$, via
\begin{equation}\label{eq:weitzenbock}
\Delta=\nabla^*\nabla + t \, \mathcal K(R,\rho),
\end{equation}
where $t\in\R$ is an appropriately chosen constant. The \emph{curvature term} $\mathcal K(R,\rho)$ is the fiberwise linear endomorphism of $E$ defined by
\begin{equation}\label{eq:krho}
\mathcal K(R,\rho)=-\sum_{a,b} R_{ab} \;\dd\rho(X_a)\dd\rho(X_b),
\end{equation}
where $\{X_a\}$ is an orthonormal basis of the Lie algebra $\mathfrak{so}(n)\cong\wedge^2\R^n$, such that the curvature operator of $(M,\g)$ is identified with $R=\sum_{a,b} R_{ab} \,X_a\otimes X_b$ by using an orthonormal frame on $M$, cf.~\cite[(1.142)]{Besse}, \cite{Hitchin15}. Note that $\mathcal K(R,\rho)$ is self-adjoint, since $R$ is self-adjoint and $\dd\rho(X)$ are skew-adjoint, as $\rho$ is orthogonal.

A fundamental observation is that \eqref{eq:krho}, being a pointwise definition, can be understood abstractly as an algebraic construction that, to each $R\in\Sym^2(\wedge^2\R^n)$, that is, to each \emph{modified algebraic curvature operator}, associates $\mathcal K(R,\rho)\in\Sym^2(E)$. In other words, there is no need for $R$ to satisfy the first Bianchi identity. The following properties of \eqref{eq:krho} are straightforward to verify.

\begin{proposition}\label{prop:Krho}
The map $\mathcal K(\cdot,\rho)\colon\Sym^2(\wedge^2 \R^n)\to\Sym^2(E)$ satisfies:
\begin{enumerate}[\rm (i)]
\item $\mathcal K(\cdot,\rho)$ is linear and $\O(n)$-equivariant;
\item $\mathcal K(R,\rho)=0$ if $\rho$ is the trivial representation (or, more generally, if $\dd \rho=0$);
\item $\mathcal K(R,\rho_1\oplus\rho_2)=\mathcal K(R,\rho_1)\oplus\mathcal  K(R,\rho_2)\in\Sym^2(E_1)\oplus\Sym^2(E_2)\subset\Sym^2(E_1\oplus E_2)$ for any two $\O(n)$-representations $\rho_i\colon\O(n)\to \O(E_i)$;
\item If $\rho$ has no trivial factors and $R$ is positive-definite, then $\mathcal K(R,\rho)$ is also positive-definite. In particular, if $R$ is positive-semidefinite, then so is $\mathcal K(R,\rho)$.
\end{enumerate}
\end{proposition}

\begin{example}\label{ex:Ricci}
Let $\rho$ be the defining $\O(n)$-representation on $\R^n$, and recall that the corresponding associated bundle is the tangent bundle $TM$.
Set $\{X_a\}$ to be the basis of $\mathfrak{so}(n)$ consisting of matrices $E_{ij}$, $1\leq i<j\leq n$, defined in the conventions in the Introduction. Given $R\in\Sym^2(\wedge^2\R^n)$, write $R_{ijkl}$ for the coefficient $R_{ab}$, where $X_a=E_{ij}$ and $X_b=E_{kl}$. Since $\dd\rho(X)=X$, we have that, for all $p\leq q$,
\begin{align*}
\langle\mathcal K(R,\rho)e_p,e_q\rangle &=-\sum_{i<j}\sum_{k<l} R_{ijkl}\langle E_{ij}E_{kl}e_p,e_q\rangle\\
&=\sum_{i<j}\sum_{k<l} R_{ijkl}\langle \delta_{lp}e_k-\delta_{kp}e_l,\delta_{jq}e_i-\delta_{iq}e_j\rangle\\
&=\sum_{i<j} \sum_{k<l} R_{ijkl} (\delta_{lp}\delta_{jq}\delta_{ki}-\delta_{lp}\delta_{iq}\delta_{kj}-\delta_{kp}\delta_{jq}\delta_{li}+\delta_{kp}\delta_{iq}\delta_{lj})\\
&=\sum_{i<p} R_{iqip}-\sum_{p<i<q} R_{iqpi}+\sum_{q<j} R_{qjpj}\\
&=\sum_{i} R_{piqi}\\
&=\Ric_R(e_p,e_q).
\end{align*}
Thus, the curvature term $\mathcal K(R,\rho)\in\Sym^2(\R^n)$ is the \emph{Ricci tensor} of $R$.
\end{example}

\subsection{Exterior and Symmetric Powers}\label{subsec:ESpowers}
The representations $\rho$ of $\O(n)$ that are most relevant for this paper are the \emph{exterior powers} $\wedge^p\R^n$ and \emph{symmetric powers} $\Sym^p\R^n$ of the defining representation of $\O(n)$ in $\R^n$, respectively given by
\begin{equation}\label{eq:ESpowers}
\begin{aligned}
\rho(A)(v_1\wedge \dots\wedge v_p) &=Av_1\wedge \dots \wedge Av_p,\\ 
\rho(A)(v_1\vee \dots\vee v_p) &=Av_1\vee \dots \vee Av_p,
\end{aligned}
\end{equation}
for all $A\in\O(n)$, $v_i\in\R^n$. The linearizations of \eqref{eq:ESpowers} are given by the Leibniz Rule:
\begin{equation}\label{eq:LeibnizRule}
\begin{aligned}
\dd\rho(X)(v_1\wedge \dots\wedge v_p) &=(X v_1)\wedge \dots \wedge v_p+\dots+v_1\wedge\dots\wedge v_{p-1}\wedge (Xv_p),\\ 
\dd\rho(X)(v_1\vee \dots\vee v_p) &=(X v_1)\vee \dots \vee v_p+\dots+v_1\vee\dots\vee v_{p-1}\vee (Xv_p),
\end{aligned}
\end{equation}
for all $X\in\mathfrak{so}(n)$. Identifying $\varphi\in\Sym^p\R^n$ with a homogeneous polynomial $\varphi\colon\R^n\to\R$, as explained in the end of the Introduction, the action of $\O(n)$ becomes
\begin{equation*}
\rho(A)(\varphi)=\varphi\circ A^{-1},
\end{equation*}
where $A\in\O(n)$ is identified with the orthogonal transformation $A\colon \R^n\to\R^n$. In this language, the linearization of $\rho$ at a matrix $X=E_{ij}\in\mathfrak{so}(n)$ becomes
\begin{equation}\label{eq:drhoEij}
\dd\rho(E_{ij})\varphi=\tfrac{\dd}{\dd t}\left(\varphi \circ e^{-t E_{ij}}\right)\big|_{t=0}=\left(x_i\tfrac{\partial}{\partial x_j}-x_j\tfrac{\partial}{\partial x_i}\right)\varphi.
\end{equation}

\begin{remark}\label{rem:trivialcases}
Although the $\O(n)$-representations $\wedge^{n-1}\R^n$ and $\wedge^1\R^n$ are not isomorphic, their linearizations agree (identifying $\wedge^{n-1}\R^n\cong\wedge^1\R^n\cong\R^n$ as vector spaces) and hence, by Example~\ref{ex:Ricci}, we have that $\mathcal K(R,\wedge^{n-1}\R^n)=\mathcal K(R,\wedge^{1}\R^n)=\mathcal K(R,\R^n)=\Ric_R$. Moreover, $\wedge^0\R^n$, $\wedge^n\R^n$, and $\Sym^0\R^n$ have trivial linearization and hence the corresponding curvature terms $\mathcal K(R,\rho)$ vanish, see Proposition~\ref{prop:Krho}.
\end{remark}

It is a well-known fact in Representation Theory that $\wedge^p \R^n$ is an irreducible $\O(n)$-representation of \emph{real type}, that is, its algebra of equivariant endomorphisms is isomorphic to $\R$. On the other hand, for $p\geq2$, the representation $\Sym^p \R^n$ decomposes as a sum of $\lceil \frac{p+1}{2}\rceil$ irreducible $\O(n)$-representations of real type
\begin{equation}\label{eq:DecompSymP}
\Sym^p\R^n\cong \Sym_0^p\R^n\oplus\Sym_0^{p-2}\R^n\oplus \cdots,
\end{equation}
where the last factor is either $\Sym_0^1\R^n\cong\R^n$ or $\Sym_0^0\R^n\cong\R$, according to the parity of $p$. 
The isomorphism between each irreducible $\Sym^k_0\R^n$ and the corresponding subspace of $\Sym^p\R^n$ is given by (polynomial) multiplication by the $(p-k)/2$ power of $r^2=\sum_{i=1}^n x_i^2\in\Sym^2\R^n$, see \cite[Ex.~19.21]{FultonHarris}. 

The vector bundles over a Riemannian manifold $(M,\g)$ associated to the $\O(n)$-representations $\wedge^p\R^n$ and $\Sym^p\R^n$ are clearly the bundles $\wedge^p TM$ and $\Sym^p TM$, whose sections are, respectively, $p$-forms and symmetric $p$-tensors on $M$.
In both of these cases, there are geometrically canonical Laplacians that satisfy the Weitzenb\"ock formula~\eqref{eq:weitzenbock} as follows, see Besse~\cite[(1.144), (1.154)]{Besse}.

In the case of $\wedge^p TM$, setting $t=2$ in \eqref{eq:weitzenbock}, one obtains the \emph{Hodge Laplacian} $\Delta=\dd \,\dd^*+\dd^*\dd$. A $p$-form $\alpha\in\wedge^p TM$ is \emph{harmonic}, that is, $\Delta\alpha=0$, if and only if it is closed ($\dd\alpha=0$) and coclosed ($\dd^*\alpha=0$). Moreover, by Hodge Theory, the $p$th Betti number of a closed manifold $M$ is precisely $b_p(M)=\dim\ker\Delta$, the dimension of the space of harmonic $p$-forms on $(M,\g)$, for any Riemannian metric $\g$.

In the case of $\Sym^p TM$, setting $t=-2$ in \eqref{eq:weitzenbock}, one obtains the \emph{Lichnerowicz Laplacian} $\Delta$. A symmetric $p$-tensor $\varphi\in\Sym^pTM$ is \emph{harmonic}, that is, $\Delta \varphi=0$, if and only if it is \emph{Killing}, that is, its symmetrized covariant derivative vanishes. In particular, if $p=1$, the space of Killing vector fields $\ker\Delta\subset \Sym^1TM$ is identified with the Lie algebra of the isometry group $\Iso(M,\g)$, provided $(M,\g)$ is complete.

\subsection{The Bochner technique}\label{subsec:Bochner}
If $(M,\g)$ is closed, the classical vanishing theorems of Bochner~\cite{bochner,bochner-yano} follow by exploiting the positivity of $t\,\mathcal K(R,\rho)$ to force the vanishing of harmonic sections $\alpha\in\ker\Delta$, by integrating \eqref{eq:weitzenbock} as follows:
\begin{equation}\label{eq:bochnertech}
0=\int_M \langle \Delta\alpha,\alpha\rangle \,\vol= \int_M \|\nabla\alpha\|^2 +t\, \langle \mathcal K(R,\rho)\alpha,\alpha\rangle\,\vol.
\end{equation}
For instance, in the cases of $\wedge^1 TM$ and $\Sym^1 TM$, using  $t=\pm2$ as appropriate and Example~\ref{ex:Ricci}, the second term in the above integrand is strictly positive for nonzero harmonic $1$-forms if $\Ric>0$, and for nonzero Killing vector fields if $\Ric<0$. This would lead to a contradiction if such harmonic sections did not vanish. Thus, if a closed manifold $(M,\g)$ satisfies $\Ric>0$, then $b_1(M)=0$; while if it satisfies $\Ric<0$, then $\Iso(M,\g)$ is finite.
Clearly, such conclusions remain true if $\Ric\geq0$ and $\Ric>0$ at some point; or, respectively, $\Ric\leq0$ and $\Ric<0$ at some point.

The identity \eqref{eq:bochnertech} also yields strong results if $t\,\mathcal K(R,\rho)$ is only nonnegative, without the strict inequality at any points of $M$. In this case, harmonic sections $\alpha\in\ker\Delta$ are parallel and satisfy $\mathcal K(R,\rho)\alpha=0$. In the above examples, this proves that if a closed manifold $(M,\g)$ satisfies $\Ric\geq0$, then $b_1(M)\leq b_1(T^n)=n$; while if it satisfies $\Ric\leq0$, then $\dim\Iso(M,\g)\leq n$. Rigidity results characterizing the equality cases in the previous statements can also be easily obtained. For more details and applications, including on noncompact manifolds, see Wu~\cite{wu}.

\section{Algebraic characterization of sectional curvature bounds}
\label{sec:characterization}

The goal of this section is to algebraically characterize $\sec_R\geq k$ and $\sec_R\leq k$ in terms of the curvature terms of Weitzenb\"ock formulae for (traceless) symmetric tensors, proving Theorem~\ref{mainthm:algebraicChar}. 
The starting point is to obtain an explicit formula for $\mathcal K(R,\Sym_0^p \R^n)$, cf.~ \cite[Prop.~6.6]{HMS15} and \cite[Thm~1.6]{DS10}.
 
\begin{proposition}\label{prop:IntegralFormula}
There is a constant $c_{p,n}>0$ such that, for all $\varphi,\psi\in\Sym^p_0\R^n$,
\begin{equation}\label{E:Symk}
\left\langle \mathcal K\big(R,\Sym_0^p \R^n\big)\varphi,\,\psi\right\rangle=c_{p,n}\int_{S^{n-1}}R\big(x,\nabla\varphi(x), x, \nabla\psi(x)\big) \,\dd x.
\end{equation}
\end{proposition}

\begin{proof}
By Schur's Lemma, as $\Sym^p_0\R^n$ is an irreducible $\O(n)$-representation of real type, all $\O(n)$-equivariant inner products on $\Sym^p_0\R^n$ are (positive) real multiples of each other.
Thus, there exists a constant $c_{p,n}>0$ such that for all $\varphi,\psi\in\Sym^p_0\R^n$,
\begin{equation}\label{eq:InnProdsSymp0}
\langle \varphi,\psi\rangle = c_{p,n} \int_{S^{n-1}} \varphi(x)\psi(x)\,\dd x.
\end{equation}

Evaluating the linearization \eqref{eq:drhoEij} at some $x=(x_1,\dots, x_n)\in\R^n$, we have
\begin{align*}
\big(\dd\rho(E_{ij})\varphi\big)(x)&=x_i\tfrac{\partial\varphi}{\partial x_j}(x)-x_j\tfrac{\partial\varphi}{\partial x_i}(x)\\
&=\langle x,e_i\rangle\langle\nabla\varphi(x),e_j\rangle-\langle x,e_j\rangle\langle\nabla\varphi(x),e_i\rangle\\
&=\langle x\wedge\nabla \varphi(x),e_i\wedge e_j\rangle\\
&=\langle x\wedge \nabla\varphi(x),E_{ij}\rangle.
\end{align*}
In particular, $\big(\dd\rho(X)\varphi\big)(x)=\langle x\wedge\nabla\varphi(x),X\rangle$ for all $X\in\mathfrak{so}(n)$. Altogether, combining this with \eqref{eq:krho} and \eqref{eq:InnProdsSymp0}, we have that for any $R=\sum_{a,b} R_{ab}\,X_a\otimes X_b$,
\begin{align*}
\left\langle \mathcal K\big(R,\Sym_0^p \R^n\big)\varphi,\,\psi\right\rangle &= \sum_{a,b} R_{ab}\langle \dd\rho(X_a)\varphi,\dd\rho(X_b)\psi\rangle\\
&=c_{p,n}\sum_{a,b} R_{ab} \int_{S^{n-1}} \big(\dd\rho(X_a)\varphi\big)(x)\big(\dd\rho(X_b)\psi\big)(x)\,\dd x\\
&=c_{p,n}\sum_{a,b} R_{ab}\int_{S^{n-1}} \langle x\wedge\nabla\varphi(x),X_a\rangle \langle x\wedge\nabla\psi(x),X_b\rangle\,\dd x\\
&=c_{p,n}\int_{S^{n-1}}\langle R,(x\wedge\nabla\varphi(x))\otimes (x\wedge\nabla\psi(x))\rangle\,\dd x\\
&=c_{p,n}\int_{S^{n-1}}R\big(x,\nabla\varphi(x), x, \nabla\psi(x)\big) \,\dd x. \qedhere
\end{align*}
\end{proof}

An interesting consequence of Proposition~\ref{prop:IntegralFormula} is that $\mathcal K(R,\Sym^p_0 \R^n)=0$ for every $R\in\wedge^4\R^n\subset\Sym^2(\wedge^2\R^n)$, a fact that also follows from Theorem~\ref{mainthm:KRWedgeSymKN}.

Although Theorem~\ref{mainthm:algebraicChar} is stated using curvature terms $\mathcal K(R,\Sym^p_0\R^n)$, $p\geq2$, note that it can be alternatively stated using $\mathcal K(R,\Sym^p\R^n)$, $p\geq2$. Indeed, by Schur's Lemma and \eqref{eq:DecompSymP}, the self-adjoint $\O(n)$-equivariant endomorphism $\mathcal K(R,\Sym^p \R^n)$ is block diagonal, with blocks corresponding to the endomorphisms $\mathcal K(R,\Sym^k_0 \R^n)$, $k=p, p-2, \ldots$. Thus, positive-semidefiniteness for all $p\geq2$ of the endomorphisms $\mathcal K(R,\Sym^p_0\R^n)$ or $\mathcal K(R,\Sym^p\R^n)$ are clearly equivalent.

\begin{proof}[Proof of Theorem \ref{mainthm:algebraicChar}]
Since $\sec_{R}$ depends linearly on $R$, it suffices to prove that $\sec_R\geq0$ if and only if $\mathcal K(R,\Sym^p_0 \R^n)$ is positive-semidefinite for all $p\geq2$.

Setting $\varphi=\psi$ in \eqref{E:Symk}, the integrand in the right-hand side becomes a sectional curvature of $R$. Thus, it follows that $\mathcal K(R,\Sym^p_0 \R^n)$, $p\geq2$, is positive-semidefinite if $\sec_R\geq0$.

In order to prove the converse, let $R\colon\wedge^2\R^n\to\wedge^2\R^n$ be an algebraic curvature operator such that $\mathcal K(R,\Sym^p_0 \R^n)$ is positive-semidefinite for all $p\geq2$.
Given any polynomial $\varphi\colon\R^n\to\R$, consider its decomposition $\varphi=\varphi^\even+\varphi^\odd$ into parts of even and odd degree. After replacing $\varphi$ with a polynomial with the same restriction to the unit sphere $S^{n-1}\subset\R^n$, we may assume that $\varphi^\even$ and $\varphi^\odd$ are homogeneous. Since the function $x\mapsto R\big(x,\nabla \varphi^\even(x),x,\nabla \varphi^\odd(x)\big)$ is odd, its integral over $S^{n-1}$ vanishes, and hence
\begin{align*}
\int_{S^{n-1}}R\big(x,\nabla \varphi(x), x, \nabla \varphi(x)\big) \,\dd x = & \int_{S^{n-1}}R\big(x,\nabla \varphi^\even(x), x, \nabla \varphi^\even(x)\big)\,\dd x\\
& +\int_{S^{n-1}}R\big(x,\nabla \varphi^\odd(x), x, \nabla \varphi^\odd(x)\big) \,\dd x.
\end{align*}
Denoting by $\varphi^\even=\varphi_p+\varphi_{p-2}+\cdots+\varphi_2+\varphi_0$ the  $\O(n)$-irreducible components of $\varphi\in\Sym^p \R^n$ according to the decomposition \eqref{eq:DecompSymP}, we claim that the mixed terms in the bilinear expansion of $\int_{S^{n-1}}R\big(x,\nabla \varphi^\even(x), x, \nabla \varphi^\even(x)\big)\,\dd x$ vanish. Indeed, if $k\neq l$  and $X,Y\in\mathfrak{so}(n)$, then $\dd\rho(X)\varphi_{2k}, \dd\rho(Y)\varphi_{2l}\in\Sym^p(\R^n)$ are in different $\O(n)$-irreducible factors. A computation similar to that in the proof of Proposition~\ref{prop:IntegralFormula} implies that, for any $R=\sum_{a,b} R_{ab}\,X_a\otimes X_b$,
\begin{multline*}
\int_{S^{n-1}}R\big(x,\nabla \varphi_{2k}(x), x, \nabla \varphi_{2l}(x)\big)\,\dd x=\\
=\sum_{a,b} R_{ab} \int_{S^{n-1}} \big(\dd\rho(X_a)\varphi_{2k}\big)(x)\big(\dd\rho(X_b)\varphi_{2l}\big)(x)\,\dd x=0,
\end{multline*}
where the last equality holds since each of the integrals vanishes due to Schur's Lemma.
Therefore, by \eqref{E:Symk}, we have
\begin{equation*}
 \int_{S^{n-1}}R\big(x,\nabla \varphi^\even(x), x, \nabla \varphi^\even(x)\big)\,\dd x=\sum_{k=0}^{p/2}\frac{1}{c_{2k,n}}\langle \mathcal K(R,\Sym_0^{2k} \R^n)\varphi_{2k},\varphi_{2k}\rangle.
\end{equation*}
A similar argument applies to $\varphi^\odd$, so we conclude that for every polynomial $\varphi$, 
\begin{equation}\label{eq:IntPhiNonneg}
\int_{S^{n-1}} R\big(x,\nabla \varphi(x),x,\nabla \varphi(x)\big)\,\dd x\geq0.
\end{equation}
Moreover, \eqref{eq:IntPhiNonneg} holds for any $\varphi\in C^\infty(\R^n)$ by density of polynomials in $C^{\infty}(\R^n)$, e.g., with respect to uniform $C^1$-convergence on compact sets.

In this situation, suppose that $R$ does not satisfy $\sec_R\geq0$. Up to changing the orthonormal basis $\{e_i\}$ of $\R^n$ and rescaling $R$, this is clearly equivalent to assuming that $\sec_R(e_1\wedge e_2)=R(e_1,e_2,e_1,e_2)=-1$. Given $\varepsilon>0$, define the test function
\begin{equation*}
f_\varepsilon\colon \R^n\to \R,\qquad 
f_\varepsilon(x)=\max \{0, \varepsilon^2-|x_2|-\|x-e_1\|^2\}.
\end{equation*}
Note that the support of $f_\varepsilon$ is contained in the $\varepsilon$-ball around $e_1\in\R^n$. Moreover, for any open neighborhood $U\subset\operatorname{supp}f_\varepsilon$ of $e_1$,
we have that
\begin{equation*}
f_\varepsilon\in W^{1,2}(\R^n)\cap C^{\infty}\big(U\setminus\{x_2=0\}\big),
\end{equation*}
that is, $f_\varepsilon$ has square-integrable weak first derivatives in $\R^n$, and is smooth at all $x=(x_1,x_2,\dots, x_n)\in U$ away from the hyperplane $x_2=0$. Moreover, at such points, its gradient can be computed as
\begin{equation*}
\nabla f_\varepsilon (x)= \frac{x_2}{|x_2|} e_2 -2(x-e_1).
\end{equation*}
In particular, if $x\in U\setminus\{x_2=0\}$ is sufficiently close to $e_1$, then $\nabla f_\varepsilon(x)$ can be made arbitrarily close to $\pm e_2$.
By continuity of the sectional curvature function of $R$, this means that $R(x,\nabla f_\varepsilon(x),x,\nabla f_\varepsilon(x))$ is arbitrarily close to $-1$.
Up to choosing a smaller $\varepsilon>0$, we may hence assume that 
\begin{equation*}
\int_{S^{n-1}}R\big(x,\nabla f_\varepsilon(x), x, \nabla f_\varepsilon(x)\big)\,\dd x<0.
\end{equation*}
By the density of $C^{\infty}(\R^n)$ in $W^{1,2}(\R^n)$, there is a sequence of smooth functions converging to $f_\varepsilon$ in $W^{1,2}(\R^n)$.
Elements in this sequence that are sufficiently close to $f_\varepsilon$ in $W^{1,2}(\R^n)$ must hence violate \eqref{eq:IntPhiNonneg}, giving the desired contradiction.
\end{proof}


\begin{remark}\label{rem:Berger}
As communicated to us by M.~Gursky, it was observed by M.~Berger, cf.~\cite[Thm.~16.9]{Besse}, that if $\mathcal K(R,\Sym^2_0 \R^n)$ is positive-semidefinite, then so is $\mathcal K(R,\Sym^1\R^n)$. Namely, since the latter is precisely the Ricci tensor of $R$, see Example~\ref{ex:Ricci}, it can be diagonalized with an orthonormal basis $\{v_i\}$ of $\R^n$, that is, $\mathcal K(R,\Sym^1\R^n)(v_i)=\Ric_R(v_i)=\lambda_i v_i$.
It follows from \eqref{eq:drhoEij} that the linearization of the $\O(n)$-representation $\Sym^2_0\R^n$ acts on $\varphi_m=v_m\vee v_m - \frac{1}{n}\,\g\in\Sym^2_0\R^n$ as
\begin{equation*}
\dd\rho(v_i\wedge v_j)(\varphi_m)=\left( x_i\tfrac{\partial}{\partial x_j}-x_j\tfrac{\partial}{\partial x_i}\right)\left(x_m^2-\tfrac{1}{n}\textstyle\sum_{i=1}^n x_i^2\right)=2\delta_{jm}x_i x_m-2\delta_{ik}x_j x_k,
\end{equation*}
where we denote by $(x_1,\dots,x_n)$ the coordinate system in $\R^n$ defined by $\{v_i\}$. Combining \eqref{eq:krho} with the above, we obtain:
\begin{multline*}
\langle \mathcal K(R,\Sym^2_0 \R^n)(\varphi_m),\varphi_m\rangle = \sum_{i<j} \sum_{k<l} R_{ijkl}\langle \dd\rho(v_i\wedge v_j)\varphi_m,\dd\rho(v_k\wedge v_l)\varphi_m\rangle\\
=4\sum_{i<j}\sum_{k<l} R_{ijkl} \langle \delta_{jm}x_i x_m - \delta_{im}x_j x_m,\delta_{lm}x_k x_m - \delta_{km}x_l x_m\rangle\\
=4\sum_{i<j}\sum_{k<l} R_{ijkl}(\delta_{jm}\delta_{lm}\delta_{ik}-\delta_{jm}\delta_{km}\delta_{il}-\delta_{im}\delta_{lm}\delta_{jk}+\delta_{im}\delta_{km}\delta_{jl})\\
=4\sum_{1\leq i<m} R_{imim}+4\sum_{m<i\leq n} R_{mimi}
=4\,\Ric_R(v_m,v_m)=4\,\lambda_m.
\end{multline*}
Thus, positive-semidefiniteness of $\mathcal K(R,\Sym^2_0 \R^n)$ clearly implies positive-semidefi\-niteness of $\mathcal K(R,\Sym^1\R^n)=\Ric_R$.
Given the above, it is natural to investigate whether positive-semidefiniteness of
$\mathcal K(R,\Sym^{p+1}_0 \R^n)$ and $\mathcal K(R,\Sym^{p}_0 \R^n)$ can also be related for $p\geq2$. In principle, this could shed light on further algebraic characterizations of $\sec\geq0$ that are more powerful than the one stated in Theorem~\ref{mainthm:algebraicChar}. 
\end{remark}

\section{Kulkarni--Nomizu algebras}
\label{sec:KNalgebras}

In this section, we recall the classical Kulkarni--Nomizu product of symmetric tensors on exterior powers, and develop an analogous product for symmetric tensors on symmetric powers. We then analyze representation theoretic aspects of the corresponding algebras, in preparation for proving Theorem~\ref{mainthm:KRWedgeSymKN}.

\subsection{Kulkarni--Nomizu products}\label{subsec:KNproducts} 
The classical Kulkarni--Nomizu product is the map 
$\owedge$ that to each $h,k\in\Sym^2(\R^n)$ associates $h\owedge k\in\Sym^2(\wedge^2\R^n)$ defined by
\begin{equation}\label{eq:KNproduct2}
\begin{aligned}
(h\owedge k)(x\wedge y,z\wedge w)&=h(x,z)k(y,w)+h(y,w)k(x,z)\\ &\quad -h(x,w)k(y,z)-h(y,z)k(x,w)
\end{aligned}
\end{equation}
on decomposable $2$-forms, and extended by linearity. 
This operation simplifies the algebraically manipulation of curvature operators, e.g., the curvature operator of a manifold with constant curvature $k$ is simply $\frac{k}{2}\,\g\owedge\g$. Moreover, it can be used to decompose the space of modified algebraic curvature operators into its $\O(n)$-irreducible components $\Sym^2(\wedge^2\R^n)=\mathcal U\oplus\mathcal L\oplus\mathcal W\oplus\wedge^4\R^n$ via successive ``divisions" by $\g$; namely $R=R_\mathcal U+R_\mathcal L+R_\mathcal W+R_{\wedge^4}$, where
\begin{equation}\label{eq:RURL}
R_\mathcal U=\frac{\scal}{2n(n-1)} \,\g\owedge\g, \qquad R_\mathcal L=\frac{1}{n-2}\, \g\owedge\left(\Ric-\frac{\scal}{n}\right)\!,
\end{equation}
and $R_\mathcal W$ and $R_{\wedge^4}$ do not have $\g$ factors, see~\cite[(1.116)]{Besse}. In this context, the curvature operator $R$ is respectively scalar flat, Einstein, conformally flat, or satisfies the first Bianchi identity if and only if $R_\mathcal U$, $R_\mathcal L$, $R_\mathcal W$, or $R_{\wedge^4}$ vanishes.

As observed by Kulkarni~\cite{Kulkarni72}, the Kulkarni--Nomizu product \eqref{eq:KNproduct2} is a special case of the product $\owedge$ in the commutative graded algebra
\begin{equation}
\mathcal C:=\bigoplus_{p=0}^n \Sym^2(\wedge^p \R^n)
\end{equation}
induced from the exterior algebra $\bigoplus_{p=0}^n \wedge^p\R^n$ as follows. For $\alpha,\beta\in\wedge^p\R^n$ and $\gamma,\delta\in\wedge^q\R^n$, let
\begin{equation*}
(\alpha \otimes \beta) \owedge (\gamma\otimes\delta) := (\alpha \wedge \gamma)\otimes (\beta\wedge\delta).
\end{equation*}
Extend the above by linearity to a product on $\bigoplus_{p=0}^n (\wedge^p \R^n)^{\otimes 2}$, and note that $\mathcal C$ is invariant under this product. Henceforth, this product $\owedge$ on $\mathcal C$ will also be called the (classical) Kulkarni--Nomizu product.

A convenient way to express the curvature term in Weitzenb\"ock formulae for symmetric tensors (used, e.g., in Theorem~\ref{mainthm:KRWedgeSymKN}) involves a symmetric analogue of the  above Kulkarni--Nomizu product. Consider the commutative graded algebra
\begin{equation}
\mathcal A:=\bigoplus_{p=0}^\infty \Sym^2(\Sym^p \R^n)
\end{equation}
endowed with the product $\ovee$ induced from the symmetric algebra $\bigoplus_{p=0}^\infty \Sym^p\R^n$ as follows.
For $\alpha,\beta\in\Sym^p\R^n$ and $\gamma,\delta\in\Sym^q\R^n$, let
\begin{equation*}
(\alpha \otimes \beta) \ovee (\gamma\otimes\delta) := (\alpha \vee \gamma)\otimes (\beta\vee\delta).
\end{equation*}
Extend the above by linearity to a product on $\bigoplus_{p=0}^\infty (\Sym^p \R^n)^{\otimes 2}$, and note that the subspace $\mathcal A$ is invariant under this product.
Due to the reducibility of $\Sym^p\R^n$, $p\geq2$, as an $\O(n)$-representation, it is useful to study the subspace
\begin{equation}
\mathcal A_0:=\bigoplus_{p=0}^\infty \Sym^2(\Sym^p_0 \R^n).
\end{equation}
Consider the projection $\pi\colon\mathcal A\to\mathcal A_0$ obtained by extending by linearity the map
\begin{equation}\label{eq:piAA0}
\pi(\varphi\otimes \psi):=\varphi_0\otimes \psi_0, \quad \varphi,\psi\in\Sym^p\R^n,
\end{equation}
to all of $\mathcal A$. Define a product on $\mathcal A_0$, also denoted by $\ovee$, by setting for all $a, b\in\mathcal A_0$,
\begin{equation}\label{eq:prodA0}
a \ovee b := \pi(a \ovee b).
\end{equation}

\begin{proposition}
In the above notation, $(\mathcal A_0,\ovee)$ is a commutative graded algebra.
\end{proposition}

\begin{proof}
Since $\mathcal A$ is a commutative graded algebra, it suffices to show that  $\ovee$ is well-defined on $\mathcal A_0$, that is, that $\ker \pi$ is a graded ideal. 
Since, for each $p$,
\begin{equation*}
\Sym^p\R^n=\Sym^p_0\R^n\oplus (r^2 \vee \Sym^{p-2}\R^n),
\end{equation*}
it follows that
\begin{equation*}
\begin{aligned}
\ker \pi&=\bigoplus_{p=0}^\infty \Sym^2(r^2\vee \Sym^{p-2}\R^n)\oplus \big(\!\Sym^p_0\R^n \otimes (r^2\vee \Sym^{p-2}\R^n)\big)\\
&=\operatorname{span}\big\{(r^2\vee \xi) \otimes \zeta + \zeta \otimes (r^2\vee \xi), \xi\in\Sym^{p-2}\R^n, \zeta \in \Sym^p\R^n, p\geq 2\big\}.
\end{aligned}
\end{equation*}
Given $a\in\mathcal A$ and $b\in\ker \pi$, we claim that $a\ovee b\in\ker\pi$.
By linearity, we may assume that
\begin{equation*}
\begin{aligned}
a&=\varphi\otimes \psi+\psi\otimes\varphi,& &\varphi,\psi\in\Sym^p\R^n,\\
b&=(r^2\vee \xi) \otimes \zeta + \zeta \otimes (r^2\vee \xi),& &\xi\in\Sym^{p-2}\R^n, \zeta \in \Sym^p\R^n.
\end{aligned}
\end{equation*}
Then
\begin{equation*}
\begin{aligned}
a\ovee b&=(r^2\vee \varphi\vee \xi) \otimes (\psi \vee \zeta) + (\varphi\vee\zeta)\otimes (r^2\vee \psi\vee\xi)  \\
&\quad +(r^2\vee \psi\vee\xi)\otimes (\varphi\vee\zeta) + (\psi\vee\zeta)\otimes (r^2\vee\varphi\vee\xi)\in\ker\pi.\qedhere
\end{aligned}
\end{equation*}
\end{proof}

A simple (yet important) observation is that the operation $(\cdot)\ovee \g$ of multiplying by $\g$ is $\O(n)$-equivariant and preserves positive-semidefiniteness of endomorphisms.

\begin{example}\label{ex:gpowers}
The metric $\g=\sum_{i=1}^n e_i\otimes e_i\in\Sym^2(\R^n)$ is simultaneously an element of the algebras $(\mathcal C,\owedge)$, $(\mathcal A,\ovee)$, and $(\mathcal A_0,\ovee)$, and its $p$th power is exactly $p!\id$.

More precisely, in the algebra $(\mathcal C,\owedge)$, cf.~\cite{Kulkarni72,Labbi15},
\begin{equation}\label{eq:gpowedge}
\begin{aligned}
\g^{\owedge p} &=\sum_{i_1,\dots, i_p=1}^n (e_{i_1}\wedge \dots \wedge e_{i_p}) \otimes (e_{i_1}\wedge \dots \wedge e_{i_p})\\
&=p!\sum_{i_1<\dots< i_p}(e_{i_1}\wedge \dots \wedge e_{i_p}) \otimes (e_{i_1}\wedge \dots \wedge e_{i_p})\\
&=p! \,\id_{\wedge^p\R^n}.
\end{aligned}
\end{equation}
To deal with the case of $(\mathcal A,\ovee)$, we need to introduce multi-indices $\ell=(\ell_1,\dots,\ell_n)$, $\ell_i\geq0$, and we write $\ell!:=\ell_1!\dots\ell_n!$, $|\ell|=\ell_1+\dots+\ell_n$, and $x^\ell:=x_1^{\ell_1} \dots x_n^{\ell_n}\in\Sym^{|\ell|}\R^n$. Note that $\|x^\ell\|^2=\ell!$, and therefore we have the analogous result
\begin{equation}\label{eq:gpovee}
\begin{aligned}
\g^{\ovee p} &=\sum_{i_1,\dots, i_p=1}^n (e_{i_1}\vee \dots \vee e_{i_p}) \otimes (e_{i_1}\vee \dots \vee e_{i_p})\\
&=\sum_{|\ell|=p} \tfrac{p!}{l!} x^\ell \otimes x^\ell\\
&=p! \,\id_{\Sym^p\R^n}.
\end{aligned}
\end{equation}
Finally, \eqref{eq:piAA0}, \eqref{eq:prodA0}, and \eqref{eq:gpovee} imply that, in $(\mathcal A_0,\ovee)$, we also have
\begin{equation}\label{eq:gpovee0}
\g^{\ovee p}=p! \,\pi(\id_{\Sym^p\R^n})=p!\,\id_{\Sym^p_0\R^n}.
\end{equation}
Henceforth, to simplify notation, we denote the above powers \eqref{eq:gpowedge}, \eqref{eq:gpovee}, and \eqref{eq:gpovee0} by $\g^p$ when it is clear from the context which of the algebras $\mathcal A$, $\mathcal A_0$, or $\mathcal C$ is meant.
\end{example}

\subsection{Representation theory}\label{subsec:reptheory}
One of the steps to prove Theorem~\ref{mainthm:KRWedgeSymKN} is to analyze the decomposition of $\Sym^2(\wedge^p \R^n)$ and $\Sym^2(\Sym^p_0\R^n)$ into irreducible $\O(n)$-representa\-tions. Namely, we show that they contain at most one factor isomorphic to each of the irreducible factors $\mathcal U$, $\mathcal L$, $\mathcal W$, and $\wedge^4\R^n$ of the $\O(n)$-representation $\Sym^2(\wedge^2\R^n)$.
Indeed, if $p=0$, then $\Sym^2(\wedge^0 \R^n)=\Sym^2(\Sym^0_0\R^n)=\mathcal U$ is the trivial representation; and if $p=1$, then
\begin{equation*}
\Sym^2(\wedge^1 \R^n)=\Sym^2(\Sym^1_0\R^n)=\Sym^2(\R^n)\cong \mathcal U\oplus\mathcal L.
\end{equation*}
Moreover, it is also easy to see that $\Sym^2(\wedge^{n-1}\R^n)\cong \mathcal U\oplus\mathcal L$ and $\Sym^2(\wedge^{n}\R^n)\cong \mathcal U$. However, this analysis is substantially more involved if $p\geq2$ in the case of $\Sym^p\R^n$, and $2\leq p\leq n-2$ in the case of $\wedge^p\R^n$, and it is carried out using the so-called \emph{Weyl's construction}, see Lemmas~\ref{lemma:onecopySym} and \ref{lemma:onecopyWedge}. We now give a brief summary of this method; for details see \cite{FultonHarris}.

Irreducible $\GL(n,\C)$-representations are labeled in terms of integer partitions $\lambda$ of $k\in\mathds N$ by applying their \emph{Schur Functor} $\Ss_\lambda$ to the defining representation $\C^n$. Each such partition $\lambda$ determines an idempotent element $c_\lambda$ in the group algebra $\C \mathfrak S_k$ called its \emph{Young symmetrizer}, where $\mathfrak S_k$ denotes the group of permutations in $k$ letters. Using the natural action of $\C \mathfrak S_k$ on the $k$th tensor power $(\C^n)^{\otimes k}$, define $\Ss_\lambda \C^n:=c_\lambda\cdot (\C^n)^{\otimes k}\subset (\C^n)^{\otimes k}$. For instance, the partition $\lambda=(1,1,\dots,1)$ of $k$, that is, $k=1+1+\dots+1$, gives rise to the exterior power $\Ss_\lambda\C^n\cong \wedge^k\C^n$, while the trivial partition $\mu=(k)$, that is, $k=k$, gives rise to the symmetric power $\Ss_\mu\C^n\cong\Sym^k\C^n$.

The restriction of the irreducible $\GL(n,\C)$-representation $\Ss_\lambda \C^n$ to $\O(n,\C)$ is, in general, reducible. 
Intersecting $\Ss_\lambda\C^n$ with the kernels of all contraction maps $(\C^n)^{\otimes k}\to (\C^n)^{\otimes (k-2)}$ yields an $\O(n,\C)$-irreducible factor, denoted by $\Ss_{[\lambda]}\C^n$. 
Similarly to the case of $\GL(n,\C)$, this gives a one-to-one correspondence between integer partitions $\lambda$ and irreducible $\O(n,\C)$-representations. 
For instance, we have that $\Ss_{[(1,1,\dots,1)]}\C^n=\Ss_{(1,1,\dots,1)}\C^n\cong\wedge^k\C^n$ and $\Ss_{[(k)]}\C^n\cong\Sym^k_0\C^n$.


It is well-known that the irreducible factors $\mathcal U$, $\mathcal L$, $\mathcal W$, and $\wedge^4\R^n$ of the $\O(n)$-representation $\Sym^2(\wedge^2\R^n)$ are of \emph{real type}, that is, the algebra of equivariant automorphisms consists of all real scalar multiples of the identity. In particular, their complexifications are $\O(n,\C)$-irreducible and hence of the form $\Ss_{[\lambda]}\C^n$ for some integer partition $\lambda$. Namely, $\mathcal U_\C \cong\C$ is the trivial representation (corresponding to the empty partition), $\mathcal L_\C\cong \Ss_{[(2)]}\C^n$ corresponds to the partition $2=2$, $\mathcal W_\C\cong \Ss_{[(2,2)]}\C^n$ corresponds to the partition $4=2+2$, and $(\wedge^4 \R^n)_\C=\wedge^4\C^n\cong\Ss_{[(1,1,1,1)]}\C^n$ corresponds to the partition $4=1+1+1+1$.


The tensor product of two Schur Functors decomposes into a sum of Schur Functors according to the \emph{Littlewood--Richardson Rule}, see \cite[(6.7)]{FultonHarris}:
\begin{equation}\label{E:LRrule}
\Ss_\lambda\otimes\Ss_\mu=\bigoplus_\nu N_{\lambda\mu\nu}\Ss_\nu.
\end{equation}
The multiplicities $N_{\lambda\mu\nu}$ (collectively called \emph{Littlewood--Richardson numbers}) are defined combinatorially from the partitions $\lambda$, $\mu$, and $\nu$, as the number of ways one may achieve the Young diagram of $\nu$ by performing a strict $\mu$-expansion to the Young diagram of $\lambda$, see \cite[p.\ 456]{FultonHarris}. 
Furthermore, in order to decompose the restriction of an irreducible $\GL(n,\C)$-representation to $\O(n,\C)$, we use \emph{Littlewood's Restriction Formula}, see \cite[(25.37)]{FultonHarris} and \cite{Littlewood44}:
\begin{equation}
\label{E:restriction}
\mathrm{Res}^{\GL(n,\C)}_{\O(n,\C)}(\Ss_\nu \C^n)=\bigoplus_{\bar{\lambda}} N_{\nu\bar{\lambda}}\Ss_{[\bar{\lambda}]}\C^n,
\end{equation}
where $N_{\nu\bar{\lambda}}=\sum_\delta N_{\delta\bar{\lambda}\nu}$ is the sum of the corresponding Littlewood--Richardson numbers over the partitions $\delta$ with even parts. 

\begin{lemma}
\label{lemma:onecopySym}
The $\O(n)$-representation $\Sym^2(\Sym_0^p \R^n)$ has exactly one irreducible factor isomorphic to each of $\mathcal U$, $\mathcal L$, and $\mathcal W$, and no irreducible factors isomorphic to $\wedge^4\R^n$, for all $p\geq2$, $n\geq4$. 
\end{lemma}

\begin{proof}
Since their complexifications are $\O(n,\C)$-irreducible, 
the number of factors isomorphic to $\mathcal U$, $\mathcal L$, $\mathcal W$, and $\wedge^4\R^n$ that appear in $\Sym^2(\Sym^p_0 \R^n)$ is equal to the number of factors isomorphic to $\mathcal U_\C$, $\mathcal L_\C$, $\mathcal W_\C$, and $\wedge^4\C^n$ that appear in $\Sym^2(\Sym^p_0 \C^n)=\Sym^2(\Ss_{[(p)]}\C^n)$, respectively. 

To simplify notation, we henceforth drop the symbol $\C^n$ and only write $\Sym^p$ for $\Sym^p \C^n$, $\Ss_\lambda$ for $\Ss_\lambda \C^n$, etc. Furthermore, we use the formalism of virtual representations, that is, sums $\oplus$ as well as differences $\ominus$ of representations.

First, recall $\Sym^p=\Ss_{[(p)]}\oplus\Ss_{[(p-2)]}\oplus\cdots=\Ss_{[(p)]}\oplus\Sym^{p-2}$, cf.~\eqref{eq:DecompSymP}. Therefore
\begin{equation}
\Sym^2(\Ss_{[(p)]})=\Sym^2(\Sym^p)\ominus \Sym^2(\Sym^{p-2})\ominus (\Ss_{[(p)]}\otimes\Sym^{p-2}). 
\end{equation}
Since the last term equals 
$\Sym^p\otimes\Sym^{p-2}\ominus\Sym^{p-2}\otimes\Sym^{p-2}$,
we have that
\begin{equation}
\label{E:S2traceless}
\begin{aligned}
\Sym^2(\Ss_{[(p)]})= & \Sym^2(\Sym^p)\ominus\Sym^2(\Sym^{p-2})\\
& \ominus (\Sym^p\otimes\Sym^{p-2})\oplus(\Sym^{p-2}\otimes\Sym^{p-2}).
\end{aligned}
\end{equation}
All the terms in the right hand side of the above equation are restrictions to $\O(n,\C)$ of $\GL(n,\C)$-representations, so that \eqref{E:LRrule} and \eqref{E:restriction} can be used to count the number of copies of the desired representations in each term. More precisely, it follows from the Littlewood--Richardson rule \eqref{E:LRrule}, see \cite[Exercise 6.16, p.~81]{FultonHarris}, that
\begin{equation*}
\begin{aligned}
&\Sym^2(\Sym^p) =\bigoplus_{\substack{0\leq a\leq p \\ p+a\text{ even} }}\Ss_{(p+a,p-a)}, 
& \Sym^2(\Sym^{p-2}) =\bigoplus_{\substack{0\leq a\leq p-2 \\ p+a\text{ even} }}\Ss_{(p+a-2,p-a-2)},\\
&\Sym^p \otimes\Sym^{p-2} = \bigoplus_{a=0}^{p-2} \Ss_{(p+a,p-a-2)}, & \Sym^{p-2} \otimes\Sym^{p-2}= \bigoplus_{a=0}^{p-2} \Ss_{(p+a-2,p-a-2)}.
\end{aligned}
\end{equation*}

In order to count factors isomorphic to $\mathcal L_\C=\Ss_{[(2)]}$ in the restriction of terms of the form $\Ss_{(k+a,k-a)}$ to $\O(n,\C)$, apply \eqref{E:restriction} with partitions $\nu=(k+a,k-a)$ and $\bar{\lambda}=(2)$. There are several cases to be analyzed, as follows.
\begin{enumerate}[\rm (i)]
\item If $k+a$ is even and $a=0$ or $a=k$, there is a \emph{unique} strict $\bar\lambda$-expansion starting from a partition $\delta$ with even parts arriving at $\nu$:
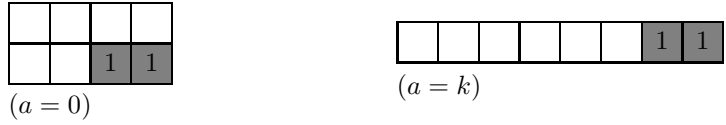
\begin{center}
\begin{minipage}{0.35\textwidth}
\ydiagram{4,4}*[*(gray) 1]{0,2+2} \\[2pt] ($a=0$)
\end{minipage} \hspace{0.5cm}
\begin{minipage}{0.35\textwidth}  \ydiagram{8}*[*(gray) 1]{6+2} \\ ($a=k$) \end{minipage}
\end{center}
\item If $k+a$ is even and $0<a<k$, there are \emph{two} strict $\bar\lambda$-expansions:
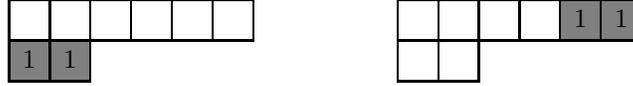
\begin{center}
\begin{minipage}{0.35\textwidth}
\ydiagram{6,2}*[*(gray) 1]{0,2} \end{minipage} \hspace{0.5cm}
\begin{minipage}{0.35\textwidth}  \ydiagram{6,2}*[*(gray) 1]{4+2,0}
\end{minipage}
\end{center}
\item If $k+a$ is odd and $0<a<k$, there is a \emph{unique} strict $\bar\lambda$-expansion:
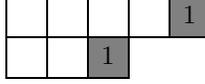
\begin{center}
\ydiagram{5,3}*[*(gray) 1]{4+1,2+1} 
\end{center}
\item If $k+a$ is odd and $a=0$, there are \emph{no} strict $\bar\lambda$-expansions by the Pieri rule.
\end{enumerate}

Analogously, the restriction of $\Ss_{(k+a,k-a-2)}$ to $\O(n,\C)$ contains two factors isomorphic to $\Ss_{[(2)]}$ if $k+a$ is even and $a<k-2$, one factor if $a=k-2$, and one factor if $k+a$ is odd.

Altogether, some elementary counting yields that the number of factors isomorphic to $\Ss_{[(2)]}$ in each term of \eqref{E:S2traceless} is given according to the following table.
\begin{center}
\begin{tabular}{|l|l|}
\hline
Term in \eqref{E:S2traceless} & Number of factors isomorphic to\rule{0pt}{2.5ex}\rule[-1.2ex]{0pt}{0pt} $\Ss_{[(2)]}$ \\
\hline \noalign{\smallskip} \hline 
\multirow{2}{*}{$\Sym^2(\Sym^p)$} & $p$ even: $1+2(\tfrac{p}{2}-1)+1=p$ \rule{0pt}{2.5ex} \\
 & $p$ odd: $2(\frac{p-1}{2})+1=p$\rule[-1.2ex]{0pt}{0pt} \rule{0pt}{2.5ex} \\
\hline
$\Sym^2(\Sym^{p-2})$\rule{0pt}{2.5ex}\rule[-1.2ex]{0pt}{0pt}  & $p-2$ (analogous to the above)\\
\hline
\multirow{2}{*}{$\Sym^p\otimes\Sym^{p-2}$}\rule{0pt}{2.5ex} & $p$ even: $1+2(\frac{p}{2}-1)+(\frac{p}{2}-1)=\frac{3p-4}{2}$\\
 & $p$ odd: $1+2(\frac{p-3}{2})+\frac{p-1}{2}=\frac{3p-5}{2}$\rule[-1.2ex]{0pt}{0pt}\rule{0pt}{2.5ex}  \\
\hline
\multirow{2}{*}{$\Sym^{p-2}\otimes\Sym^{p-2}$}\rule{0pt}{2.5ex} & $p$ even: $1+2(\frac{p}{2}-2)+1+(\frac{p}{2}-1)=\frac{3p-6}{2}$ \\
& $p$ odd: $1+2(\frac{p-3}{2})+\frac{p-3}{2}=\frac{3p-7}{2}$\rule[-1.2ex]{0pt}{0pt} \rule{0pt}{2.5ex} \\
\hline
\end{tabular}
\end{center}
Therefore, combining the quantities in the above table according to \eqref{E:S2traceless}, we have
\begin{equation*}
\begin{aligned}
p-(p-2)-\tfrac{3p-4}{2}+\tfrac{3p-6}{2} &= 1, \quad \text{ if } p \text{ is even,}\\
p-(p-2)-\tfrac{3p-5}{2}+\tfrac{3p-7}{2} &= 1, \quad \text{ if } p \text{ is odd.}
\end{aligned}
\end{equation*}
In other words, there is exactly one factor isomorphic to $\mathcal L_\C=\Ss_{[(2)]}$ in $\Sym^2(\Ss_{[(p)]})$.

We proceed in a similar fashion to count factors isomorphic to $\mathcal W_\C=\Ss_{[(2,2)]}$. Namely, we apply \eqref{E:restriction} with $\nu=(k+a,k-a)$ and $\bar\lambda=(2,2)$. Again, there are different cases to be analyzed.
\begin{enumerate}[\rm (i)]
\item If $k+a$ is even and $0\leq a<k$, there is a \emph{unique} strict $\bar\lambda$-expansion starting from a partition $\delta$ with even parts arriving at $\nu$:
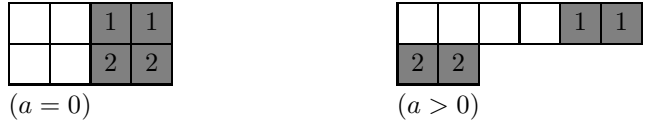
\begin{center}
\begin{minipage}{0.35\textwidth}
 \ydiagram{4,4}*[*(gray) 1]{2+2,0}*[*(gray) 2]{0,2+2} \\[2pt] ($a=0$)
\end{minipage}
\hspace{0.5cm} 
\begin{minipage}{0.35\textwidth}
\ydiagram{6,2}*[*(gray) 1]{4+2}*[*(gray) 2]{0,2} \\[2pt] ($a>0$)
\end{minipage}
\end{center}
\item If $k+a$ is even and $a=k$, or $k+a$ is odd, then there are \emph{no} strict $\bar\lambda$-expansions.
\end{enumerate}

Analogously, the restriction of $\Ss_{(k+a,k-a-2)}$ to $\O(n,\C)$ contains one factor isomorphic to $\Ss_{[(2,2)]}$ if $k+a$ is even and $a<k-2$, and no factors otherwise.

Thus, the number of factors isomorphic to $\Ss_{[(2,2)]}$ in each term of \eqref{E:S2traceless} is given according to the following table.
\begin{center}
\begin{tabular}{|l|l|}
\hline
Term in \eqref{E:S2traceless} & Number of factors\rule[-1.2ex]{0pt}{0pt} \rule{0pt}{2.5ex}  $\Ss_{[(2,2)]}$ \\
\hline \noalign{\smallskip} \hline 
\multirow{2}{*}{$\Sym^2(\Sym^p)$} & $p$ even: $\frac{p}{2}$\rule{0pt}{2.5ex}  \\
 & $p$ odd: $\frac{p-1}{2}$\rule[-1.2ex]{0pt}{0pt} \rule{0pt}{2.5ex} \\
\hline
\multirow{2}{*}{$\Sym^2(\Sym^{p-2})$, $\Sym^p\otimes\Sym^{p-2}$, $\Sym^{p-2}\otimes\Sym^{p-2}$} & $p$ even: $\frac{p-2}{2}$\rule{0pt}{2.5ex}  \\
 & $p$ odd: $\frac{p-3}{2}$\rule{0pt}{2.5ex} \rule[-1.2ex]{0pt}{0pt}  \\
\hline
\end{tabular}
\end{center}
Therefore, combining the quantities in the above table according to \eqref{E:S2traceless}, it follows that there is exactly one factor isomorphic to $\mathcal W_\C=\Ss_{[(2,2)]}$ in $\Sym^2(\Ss_{[(p)]})$.

To tackle the case $\wedge^4\C^n=\Ss_{[(1,1,1,1)]}$, set $\bar{\lambda}=(1,1,1,1)$ in \eqref{E:restriction}. Any strict $\bar{\lambda}$-expansion has at least four parts, and therefore $N_{\delta\bar{\lambda}\nu}=0$ for any partition $\nu$ with less than four parts. Hence, there are no factors isomorphic to $\wedge^4\C^n$ in $\Sym^2(\Ss_{[(p)]})$. 

Finally, since $\Ss_{[(p)]}$ is irreducible, $\Sym^2(\Ss_{[(p)]})$ contains exactly one copy of the trivial representation $\mathcal U_\C\cong\C$, corresponding to multiples of $\operatorname{Id}\in \Sym^2(\Ss_{[(p)]})$.
\end{proof}

\begin{lemma}
\label{lemma:onecopyWedge}
The $\O(n)$-representation $\Sym^2(\wedge^p\R^n)$ has exactly one irreducible factor isomorphic to each of $\mathcal U$, $\mathcal L$, $\mathcal W$, and $\wedge^4\R^n$, for all $2\leq p\leq n-2$, $n\geq4$.
\end{lemma}

\begin{proof}
As in the proof of Lemma~\ref{lemma:onecopySym}, it suffices to count $\O(n,\C)$-irreducible factors isomorphic to $\mathcal U_\C$, $\mathcal L_\C$, $\mathcal W_\C$, and $\wedge^4\C^n$ in $\Sym^2(\wedge^p\C^n)=\Sym^2(\Ss_{[(1,1,\ldots,1)]}\C^n)$.
By the Littlewood--Richardson rule \eqref{E:LRrule}, see \cite[Exercise 6.16, p.~81]{FultonHarris}, we have
\begin{equation}\label{E:LR-Wk}
\Sym^2(\wedge^p) =\bigoplus_{\substack{0\leq a\leq p \\ a\text{ even}}}\Ss_{\nu_a},
\end{equation}
where $\nu_a=(2,2,\ldots,2,1,1,\ldots,1)$ is the partition of $2p$ into $(p-a)$ factors $2$ and $2a$ factors $1$.

Applying the Littlewood restriction rule \eqref{E:restriction} with $\nu=\nu_a$ and $\bar\lambda=(2)$, respectively $\bar\lambda=(2,2)$, it follows that $\Ss_{\nu_a}$ contains a unique factor isomorphic to $\Ss_{[(2)]}$, respectively $\Ss_{[(2,2)]}$, if $a=0$; and no such factors if $a>0$.
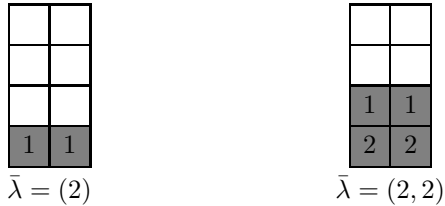
\begin{center}
\begin{minipage}{0.3\textwidth}
\begin{center}
\ydiagram{2,2,2,2}*[*(gray) 1]{0,0,0,2} \\[2pt] $\bar\lambda=(2)$
\end{center}
\end{minipage}
\hspace{0.5cm} 
\begin{minipage}{0.3\textwidth}
\begin{center}
 \ydiagram{2,2,2,2}*[*(gray) 1]{0,0,2,0}*[*(gray) 2]{0,0,0,2} \\[2pt] $\bar\lambda=(2,2)$
 \end{center}
\end{minipage}
\end{center}

Setting $\bar\lambda=(1,1,1,1)$ in \eqref{E:restriction}, it follows that $\Ss_{\nu_a}$ contains a unique factor isomorphic to $\Ss_{[(1,1,1,1)]}$ if $a=2$ and no such factors otherwise.
\begin{center}
\ydiagram{2,2,1,1,1,1}*[*(gray) 1]{0,0,1,0,0,0}*[*(gray) 2]{0,0,0,1,0,0}*[*(gray) 3]{0,0,0,0,1,0}*[*(gray) 4]{0,0,0,0,0,1}
\end{center}
Indeed, note that $N_{\delta\bar\lambda\nu_a}=0$ unless $a=2$ and $\delta=(2,\ldots,2)$ is the partition of $2p-4$ into $p-2$ equal terms, in which case $N_{\delta\bar\lambda\nu_a}=1$. Here, we have used the fact that a strict $\bar\lambda$-expansion of $\delta$ has precisely $4$ odd parts.

Thus, according to \eqref{E:LR-Wk}, it follows that $\Sym^2(\wedge^p)$ contains exactly one factor isomorphic to each of $\mathcal L_\C$, $\mathcal W_\C$, and $\wedge^4\C^n$. Finally, since $\Ss_{[(1,1,\dots,1)]}$ is irreducible, $\Sym^2(\Ss_{[(1,1,\dots,1)]})$ contains exactly one copy of the trivial representation $\mathcal U_\C\cong\C$, corresponding to multiples of $\operatorname{Id}\in \Sym^2(\Ss_{[(1,1,\dots,1)]})$.
\end{proof}

%

\section{Curvature terms for alternating and symmetric tensors}
\label{sec:proofThmB}

In this section, we give a proof of Theorem~\ref{mainthm:KRWedgeSymKN} by computing the curvature terms $\mathcal K(R,\wedge^p\R^n)$ and $\mathcal K(R,\Sym^p_0\R^n)$ of the Weitzenb\"ock formulae for alternating and symmetric tensors using the Kulkarni--Nomizu algebras studied in the previous section.
As mentioned in the Introduction, the formula for $\mathcal K(R,\wedge^p\R^n)$ under the assumption that $R_{\wedge^4}=0$ was previously obtained by Labbi~\cite[Prop.~4.2]{Labbi15}, however we provide here a new and independent proof for the sake of completeness.

\begin{proof}[Proof of Theorem~\ref{mainthm:KRWedgeSymKN}]
We begin by analyzing the case of the $\O(n)$-representations $\rho$ of the form $\Sym^p_0\R^n$, $p\geq2$. It follows from Proposition~\ref{prop:Krho} and $\O(n)$-equivariance of $(\cdot)\ovee \g\colon\mathcal A_0\to\mathcal A_0$ that the linear maps
\begin{equation}\label{eq:KSymId}
\begin{aligned}
\mathcal K(\cdot,\Sym^p_0\R^n)\colon &\Sym^2(\wedge^2\R^n) \to \Sym^2(\Sym^p_0\R^n)\\
\mathcal K(\cdot,\Sym^2_0\R^n)\ovee \g^{p-2}\colon &\Sym^2(\wedge^2\R^n) \to \Sym^2(\Sym^p_0\R^n)
\end{aligned}
\end{equation}
are $\O(n)$-equivariant. Thus, their restrictions to the $\O(n)$-irreducible factors in the decomposition $\Sym^2(\wedge^2\R^n)=\mathcal U\oplus\mathcal L\oplus\mathcal W\oplus\wedge^4\R^n$ are themselves $\O(n)$-equivariant, and their images are contained in the corresponding $\O(n)$-irreducible factors of $\Sym^2(\Sym^p_0\R^n)$. According to Lemma~\ref{lemma:onecopySym}, there are no $\O(n)$-irreducible factors isomorphic to $\wedge^4\R^n$ in $\Sym^2(\Sym^p_0\R^n)$, and hence the restrictions of the maps \eqref{eq:KSymId} to $\wedge^4\R^n\subset\Sym^2(\wedge^2\R^n)$ vanish identically by Schur's Lemma. Moreover, by Lemma~\ref{lemma:onecopySym}, there is a unique $\O(n)$-irreducible factor isomorphic to each of $\mathcal U$, $\mathcal L$, and $\mathcal W$ in $\Sym^2(\Sym^p_0\R^n)$. Provided that the restrictions of $\mathcal K(\cdot,\Sym^2_0\R^n)\ovee \g^{p-2}$ to such irreducibles do not vanish, it follows from Schur's Lemma that, since they are of real type, there exist real constants $A_{p,n}$, $B_{p,n}$, and $C_{p,n}$, such that
\begin{equation*}
\begin{aligned}
\mathcal K(R,\Sym^p_0\R^n) &=A_{p,n} \, \mathcal K(R_\mathcal U,\Sym^2_0\R^n)\ovee \g^{p-2} + B_{p,n} \, \mathcal K(R_\mathcal L,\Sym^2_0\R^n)\ovee \g^{p-2} \\
&\quad  + C_{p,n} \, \mathcal K(R_\mathcal W,\Sym^2_0\R^n)\ovee \g^{p-2}
\end{aligned}
\end{equation*}
where $R=R_{\mathcal U}+R_{\mathcal L}+R_{\mathcal W}+R_{\wedge^4}$ are the components of $R$. By evaluating the maps in \eqref{eq:KSymId} at convenient choices of $R\in\Sym^2(\wedge^2\R^n)$ and $\varphi_p\in\Sym^p_0\R^n$, we simultaneously show that $\mathcal K(\cdot,\Sym^2_0\R^n)\ovee \g^{p-2}$ restricted to each of $\mathcal U$, $\mathcal L$, and $\mathcal W$ is nonzero, and compute $A_{p,n}$, $B_{p,n}$, and $C_{p,n}$, proving the desired formula.

In what follows, using the conventions established in the Introduction, we fix
\begin{equation*}
\varphi_p(x) := \re \left(x_1+ \ii x_2\right)^p \in\Sym^p_0 \R^n.
\end{equation*}
In order to simplify computations, we use complex coordinates
\begin{equation*}
z=x_1+\ii x_2, \quad \overline z = x_1-\ii x_2,
\end{equation*}
with respect to which $\varphi_p(x)=\re z^p=\frac{z^p+\overline z^p}{2}$.
Recall that, according to our conventions, the inner product on $\Sym^p_0\R^n$ is computed as $\langle \varphi,\psi\rangle=\widehat\varphi(\psi)$, where $\widehat\varphi$ is the \emph{dual} differential operator to $\varphi$. In complex coordinates, note that the duals to $z$ and $\overline z$ are respectively $\widehat z= 2\tfrac{\partial}{\partial \overline z}$ and $\widehat{\overline z}= 2\tfrac{\partial}{\partial  z}$. For example, we may compute
\begin{equation}\label{eq:varphipnorm}
\|\varphi_p\|^2=\widehat{\varphi_p}(\varphi_p)=\tfrac{\widehat{z^p+\overline{z}^p}}{2}\left(\tfrac{z^p+\overline{z}^p}{2}\right)=2^{p-1}\, p!.
\end{equation}

First, we consider the irreducible factor $\mathcal U\cong\R$. Setting
\begin{equation}\label{eq:RU}
R_\mathcal U=\tfrac12 \g \owedge \g=\sum_{1\leq i <j\leq n} E_{ij}\otimes E_{ij},
\end{equation}
it follows from \eqref{eq:krho} that
\begin{equation}\label{eq:KRuSymp}
\langle \mathcal K(R_\mathcal U,\Sym^p_0\R^n)\varphi,\varphi\rangle = -\sum_{1\leq i <j\leq n} \langle  \dd\rho(E_{ij})^2\varphi,\varphi\rangle=\sum_{1\leq i <j\leq n} \|\dd\rho(E_{ij})\varphi\|^2.
\end{equation}
Note that $\dd\rho(E_{ij})\varphi_p=0$ for $3\leq i<j\leq n$, while
\begin{equation*}
\begin{aligned}
&\dd\rho(E_{12})\varphi_p = \left(x_1\tfrac{\partial}{\partial x_2}-x_2\tfrac{\partial}{\partial x_1}\right)\tfrac{z^p+\overline z^p}{2}= \ii \left(z\tfrac{\partial}{\partial z}-\overline z \tfrac{\partial}{\partial \overline z} \right)\tfrac{z^p+\overline z^p}{2}= -p\im z^p,\\
&\dd\rho(E_{1j})\varphi_p  = \left(x_1\tfrac{\partial}{\partial x_j}-x_j\tfrac{\partial}{\partial x_1}\right)\tfrac{z^p+\overline z^p}{2} = -x_j \left(\tfrac{\partial}{\partial z}+\tfrac{\partial}{\partial \overline z} \right)\tfrac{z^p+\overline z^p}{2}=- p\, x_j \re z^{p-1},\\
&\dd\rho(E_{2j})\varphi_p  = \left(x_2\tfrac{\partial}{\partial x_j}-x_j\tfrac{\partial}{\partial x_2}\right)\tfrac{z^p+\overline z^p}{2} = -\ii x_j \left(\tfrac{\partial}{\partial z}-\tfrac{\partial}{\partial \overline z} \right)\tfrac{z^p+\overline z^p}{2}= p\, x_j \im z^{p-1},
\end{aligned}
\end{equation*}
for $3\leq j\leq n$. Hence, the square norms of the above can be computed as
\begin{equation}\label{eq:normsdrhophip}
\begin{aligned}
\|\dd\rho(E_{12})\varphi_p\|^2 &=p^2 \tfrac{\widehat{({z}^p -{\overline z}^p})}{2\ii}
\left(\tfrac{{z}^p -{\overline z}^p}{2\ii}\right)\\
&= -\ii \, 2^{p-1} \, p^2 \left(\left(\tfrac{\partial}{\partial \overline z} \right)^p - \left(\tfrac{\partial}{\partial z} \right)^p \right)\left(\tfrac{{z}^p -{\overline z}^p}{2\ii}\right) \\
&=2^{p-1}\, p^2 \, p!,\\
\|\dd\rho(E_{1j})\varphi_p\|^2=\|\dd\rho(E_{2j})\varphi_p\|^2 &=2^{p-2}\, p \, p!, \quad 3\leq j\leq n,\\
\|\dd\rho(E_{ij})\varphi_p\|^2 &=0, \quad 3\leq i< j \leq n.
\end{aligned}
\end{equation}
Altogether, by \eqref{eq:KRuSymp}, we have that
\begin{equation}\label{eq:KRuSympSum}
\begin{aligned}
\langle \mathcal K(R_\mathcal U,\Sym^p_0\R^n)\varphi_p,\varphi_p\rangle &= 2^{p-1}\, p^2 \, p! + 2(n-2)(2^{p-2}\, p \, p!)\\
&=(n+p-2)\, 2^{p-1} \,p^2 \, (p-1)!.
\end{aligned}
\end{equation}
In particular, $\langle\mathcal K(R_\mathcal U,\Sym^2_0\R^n)\varphi_p,\varphi_p\rangle=8n$. Combining this with \eqref{eq:varphipnorm}, Example~\ref{ex:gpowers}, and $\O(n)$-equivariance, we have that $\mathcal K(R_\mathcal U,\Sym^2_0\R^n)=n\, \g\ovee\g$, and hence 
\begin{equation}\label{eq:KRuOutrolado}
\langle (\mathcal K(R_\mathcal U,\Sym^2_0\R^n)\ovee \g^{p-2}) \varphi_p,\varphi_p\rangle= n \, 2^{p-1} \, (p!)^2.
\end{equation}
Therefore, by \eqref{eq:KRuSympSum} and \eqref{eq:KRuOutrolado}, we conclude that $A_{p,n}=\tfrac{n+p-2}{n(p-1)!}$.

Second, consider the irreducible factor $\mathcal L\cong\Sym^2_0\R^n$. Setting
\begin{equation}\label{eq:RL}
R_\mathcal L=\operatorname{diag}(1,0,\dots,0,-1) \owedge \g=\sum_{2\leq j\leq n-1} E_{1j}\otimes E_{1j} -\sum_{2\leq i\leq n-1} E_{in}\otimes E_{in},
\end{equation}
it follows from \eqref{eq:krho} and \eqref{eq:normsdrhophip} that
\begin{align}\label{eq:KRLSymP}
\langle \mathcal K(R_\mathcal L,\Sym^p_0\R^n) \varphi_p,\varphi_p\rangle &=\sum_{2\leq j\leq n-1} \|\dd\rho(E_{1j})\varphi_p\|^2-\sum_{2\leq i\leq n-1} \|\dd\rho(E_{in})\varphi_p\|^2\nonumber \\
&=2^{p-1}\,p^2\,p!+(n-3)(2^{p-2}\, p\, p!)-2^{p-2}\,p\,p!\\
&=(n+2p-4)\, 2^{p-2}\, p^2\, (p-1)!.\nonumber
\end{align}
On the other hand, to compute $\langle (\mathcal K(R_\mathcal L,\Sym^2_0\R^n)\ovee \g^{p-2}) \varphi_p,\varphi_p\rangle$, we need to recognize $\dd\rho(E_{ij})^2\colon \Sym^2\R^n\to\Sym^2\R^n$ as an element of $\Sym^2(\Sym^2\R^n)$. Applying \eqref{eq:drhoEij} twice to each element of the orthonormal basis of $\Sym^2\R^n$, we obtain
\begin{equation}\label{eq:drhoEijSym2}
\dd\rho(E_{ij})^2=x_i^2\otimes x_j^2+x_j^2\otimes x_i^2 - \sum_{k=1}^n (x_i x_k\otimes x_i x_k+x_j x_k\otimes x_j x_k)-2x_ix_j\otimes x_ix_j.
\end{equation}
Combining \eqref{eq:RL}, \eqref{eq:drhoEijSym2}, Example~\ref{ex:gpowers}, and the fact that $\varphi_p(x)$ depends only on $x_1$ and $x_2$, we have that
\begin{equation*}
\begin{aligned}
\langle (\mathcal K(R_\mathcal L,\Sym^2_0\R^n)\ovee \g^{p-2}) \varphi_p,\varphi_p\rangle &=\langle (F\ovee G)\varphi_p,\varphi_p\rangle,
\end{aligned}
\end{equation*}
where
\begin{equation}\label{eq:F1F2def}
\begin{aligned}
F &= -x_1^2\otimes x_2^2-x_2^2\otimes x_1^2+n(x_1x_2\otimes x_1x_2)+(n-2)x_1^2\otimes x_1^2,\\
G &=\sum_{k=0}^{p-2} \binom{p-2}{k} \, x_1^{k} x_2^{p-k-2}\otimes x_1^{k} x_2^{p-k-2}.
\end{aligned}
\end{equation}
Expanding $F\ovee G$ according to the definition of $\ovee$ yields
\begin{equation}\label{eq:F1oveeF2}
F\ovee G=\sum_{k=0}^{p-2} \binom{p-2}{k}\, H_k,
\end{equation}
where
\begin{equation*}
\begin{aligned}
H_k &= -x_1^{k+2}x_2^{p-k-2}\otimes x_1^k x_2^{p-k}-x_1^k x_2^{p-k}\otimes x_1^{k+2}x_2^{p-k-2}\\
&\quad +n \, x_1^{k+1}x_2^{p-k-1}\otimes x_1^{k+1}x_2^{p-k-1}+(n-2)\, x_1^{k+2}x_2^{p-k-2}\otimes x_1^{k+2}x_2^{p-k-2}.
\end{aligned}
\end{equation*}
Using complex coordinates, the first term of $H_k$ acts on $\varphi_p$ as
\begin{equation}\label{eq:CoeffHk0}
\begin{aligned}
&\left\langle \big(-x_1^{k+2}x_2^{p-k-2}\otimes x_1^k x_2^{p-k}\big)\varphi_p,\varphi_p\right\rangle= \\
&\quad=\left\langle -x_1^{k+2}x_2^{p-k-2}\left(\tfrac{\partial}{\partial z}+\tfrac{\partial}{\partial \overline z}\right)^k \left(\ii \tfrac{\partial}{\partial z}-\ii\tfrac{\partial}{\partial \overline z}\right)^{p-k}\left(\tfrac{z^p+\overline{z}^p}{2}\right),\tfrac{z^p+\overline{z}^p}{2}\right\rangle\\
&\quad= -p! \ii^{p-k} \left(\tfrac{1+(-1)^{p-k}}{2}\right)
\left(\tfrac{\partial}{\partial z}+\tfrac{\partial}{\partial \overline z}\right)^{k+2} \left(\ii \tfrac{\partial}{\partial z}-\ii\tfrac{\partial}{\partial \overline z}\right)^{p-k-2}\left(\tfrac{z^p+\overline{z}^p}{2}\right)\\
&\quad=(-1)^{p-k}(p!)^2\left(\tfrac{1+(-1)^{p-k}}{2}\right).
\end{aligned}
\end{equation}
Analogously, the second, third, and fourth terms of $H_k$ act on $\varphi_p$ respectively as
\begin{align}
\left\langle \big(-x_1^k x_2^{p-k}\otimes x_1^{k+2}x_2^{p-k-2}\big)\varphi_p,\varphi_p\right\rangle &= (-1)^{p-k}(p!)^2\left(\tfrac{1+(-1)^{p-k}}{2}\right),\label{eq:CoeffHk}  \\
\left\langle \big(n \, x_1^{k+1}x_2^{p-k-1}\otimes x_1^{k+1}x_2^{p-k-1}\big)\varphi_p,\varphi_p\right\rangle &= n(-1)^{p-k-1}(p!)^2\left(\tfrac{1+(-1)^{p-k-1}}{2}\right), \nonumber \\
\left\langle \big((n-2)\, x_1^{k+2}x_2^{p-k-2}\otimes x_1^{k+2}x_2^{p-k-2}\big)\varphi_p,\varphi_p\right\rangle &=(n-2)(-1)^{p-k}(p!)^2\left(\tfrac{1+(-1)^{p-k}}{2}\right).\nonumber
\end{align}
Altogether, summing over $k$ according to \eqref{eq:F1oveeF2}, we obtain
\begin{equation}\label{eq:KRLOutrolado}
\langle (\mathcal K(R_\mathcal L,\Sym^2_0\R^n)\ovee \g^{p-2}) \varphi_p,\varphi_p\rangle=n \,2^{p-2}(p!)^2.
\end{equation}
Therefore, by \eqref{eq:KRLSymP} and \eqref{eq:KRLOutrolado}, we conclude that $B_{p,n}=\frac{n+2p-4}{n(p-1)!}$.

Third, consider the irreducible factor $\mathcal W$. It can be checked that
\begin{equation}\label{eq:RW}
R_\mathcal W=(E_{12}+E_{34})\otimes (E_{12}+E_{34}) - (E_{13}-E_{24})\otimes (E_{13}-E_{24})
\end{equation}
is orthogonal to $\mathcal U\oplus\mathcal L\oplus\wedge^4\R^n$ and hence belongs to $\mathcal W$. It follows from \eqref{eq:krho} and the fact that $\varphi_p(x)$ depends only on $x_1$ and $x_2$ that
\begin{equation*}
\begin{aligned}
\langle \mathcal K(R_\mathcal W,\Sym^p_0\R^n) \varphi_p,\varphi_p\rangle &=\|\dd\rho(E_{12})\varphi_p\|^2-\|\dd\rho(E_{13})\varphi_p\|^2-\|\dd\rho(E_{24})\varphi_p\|^2\\
&\quad +2\langle \dd\rho(E_{13})\varphi_p,\dd\rho(E_{24})\varphi_p\rangle.
\end{aligned}
\end{equation*}
The last term above vanishes because $\dd\rho(E_{13})\varphi_p$ is divisible by $x_3$, while $\dd\rho(E_{24})\varphi_p$ does not depend on $x_3$. From \eqref{eq:normsdrhophip}, we have
\begin{equation}\label{eq:KRWSymP}
\begin{aligned}
\langle \mathcal K(R_\mathcal W,\Sym^p_0\R^n) \varphi_p,\varphi_p\rangle&= 2^{p-1}\,p^2\,p!-2(2^{p-2}\,p\,p!) \\
&=(2p-2)\, 2^{p-2}\, p^2\, (p-1)!.
\end{aligned}
\end{equation}
The computation of $\langle (\mathcal K(R_\mathcal W,\Sym^2_0\R^n)\ovee \g^{p-2}) \varphi_p,\varphi_p\rangle$ is analogous to the above case pertaining to $\mathcal L$. Namely,
\begin{equation*}
\langle (\mathcal K(R_\mathcal W,\Sym^2_0\R^n)\ovee \g^{p-2}) \varphi_p,\varphi_p\rangle=\langle(F'\ovee G)\varphi_p,\varphi_p\rangle,
\end{equation*}
where $G$ is defined in \eqref{eq:F1F2def}, and $F'=x_1^2\otimes x_2^2 + x_2^2\otimes x_1^2 -2 x_1x_2\otimes x_1x_2$.
Expanding $F'\ovee G$ according to the definition of $\ovee$ yields
\begin{equation*}
F'\ovee G=\sum_{k=0}^{p-2} \binom{p-2}{k}\, H'_k,
\end{equation*}
where
\begin{equation*}
\begin{aligned}
H'_k &= -x_1^{k+2}x_2^{p-k-2}\otimes x_1^k x_2^{p-k}-x_1^k x_2^{p-k}\otimes x_1^{k+2}x_2^{p-k-2}\\
&\quad +2 \, x_1^{k+1}x_2^{p-k-1}\otimes x_1^{k+1}x_2^{p-k-1}.
\end{aligned}
\end{equation*}
Thus, from \eqref{eq:CoeffHk0} and \eqref{eq:CoeffHk}, summing over $k$, we obtain
\begin{equation}\label{eq:KRWOutrolado}
\langle (\mathcal K(R_\mathcal W,\Sym^2_0\R^n)\ovee \g^{p-2}) \varphi_p,\varphi_p\rangle=2^{p-1}(p!)^2.
\end{equation}
Therefore, by \eqref{eq:KRWSymP} and \eqref{eq:KRLOutrolado}, we conclude that $C_{p,n}=\frac{1}{(p-2)!}$.

We now turn to the case of the $\O(n)$-representations $\rho$ of the form $\wedge^p\R^n$, $2\leq p\leq n-2$. It follows from Proposition~\ref{prop:Krho} and $\O(n)$-equivariance of $(\cdot)\owedge\g\colon\mathcal C\to\mathcal C$ that the linear maps
\begin{equation}\label{eq:KWedgeId}
\begin{aligned}
\mathcal K(\cdot,\wedge^p\R^n)\colon &\Sym^2(\wedge^2\R^n) \to \Sym^2(\wedge^p\R^n)\\
(\cdot)\owedge \g^{p-2}\colon &\Sym^2(\wedge^2\R^n) \to \Sym^2(\wedge^p\R^n)
\end{aligned}
\end{equation}
are $\O(n)$-equivariant. Using Lemma~\ref{lemma:onecopyWedge} and reasoning as in the previous case, it follows that there exist real constants $A'_{p,n}$, $B'_{p,n}$, $C'_{p,n}$, and $D'_{p,n}$ such that
\begin{equation*}
\begin{aligned}
\mathcal K(R,\wedge^p\R^n) &=A'_{p,n} \, R_\mathcal U\owedge \g^{p-2} + B'_{p,n} \, R_\mathcal L\owedge \g^{p-2} \\
&\quad  + C'_{p,n} \, R_\mathcal W\owedge \g^{p-2} +  D'_{p,n} \, R_{\wedge^4}\owedge \g^{p-2}
\end{aligned}
\end{equation*}
where $R=R_{\mathcal U}+R_{\mathcal L}+R_{\mathcal W}+R_{\wedge^4}$ are the components of $R$. By evaluating the maps in \eqref{eq:KWedgeId} at convenient choices of $R\in\Sym^2(\wedge^2\R^n)$ and $\beta_p,\gamma_p\in\wedge^p \R^n$, we simultaneously show that $(\cdot)\owedge \g^{p-2}$ restricted to each of $\mathcal U$, $\mathcal L$, $\mathcal W$, and $\wedge^4 \R^n$ is nonzero, and compute $A'_{p,n}$, $B'_{p,n}$, $C'_{p,n}$, and $D'_{p,n}$ proving the desired formula.
%
%
%
Let
\begin{equation*}
\beta_p=e_1\wedge\dots\wedge e_p\in\wedge^p \R^n, \qquad \gamma_p=(e_1\wedge e_2 + e_3\wedge e_4)\wedge \delta_p \in\wedge^p\R^n,
\end{equation*}
where $\delta_p=e_5\wedge\dots\wedge e_{p+2}\in\wedge^{p-2}\R^n$.

Note that
\begin{equation}\label{eq:SqNormsDRhoWedgep}
\begin{aligned}
\|\dd\rho(E_{ij})\beta_p\|^2&=1, \quad 1\leq i\leq p<j\leq n,\\
\|\dd\rho(E_{ij})\beta_p\|^2&=0, \quad \text{otherwise.}
\end{aligned}
\end{equation}

Using the curvature operator $R_\mathcal U$ in \eqref{eq:RU} and \eqref{eq:SqNormsDRhoWedgep}, we have
\begin{equation}
\langle\mathcal K(R_\mathcal U,\wedge^p\R^n)\beta_p,\beta_p\rangle = \sum_{1\leq i<j\leq n} \|\dd\rho(E_{ij})\beta_p\|^2=p(n-p).
\end{equation}

Using the definition of $\owedge$ and Example~\ref{ex:gpowers}, we have that
\begin{equation}
\langle (R_\mathcal U\owedge \g^{p-2})\beta_p,\beta_p\rangle=\tfrac12\langle (\g^p)\beta_p,\beta_p\rangle=\tfrac12 p!.
\end{equation}
Thus, we conclude that $A'_{p,n}=\tfrac{2(n-p)}{(p-1)!}$.

Using the curvature operator $R_\mathcal L$ in \eqref{eq:RL}, by \eqref{eq:krho} and \eqref{eq:SqNormsDRhoWedgep}, we have
\begin{equation*}
\begin{aligned}
\langle\mathcal K(R_\mathcal L,\wedge^p\R^n)\beta_p,\beta_p\rangle &= \sum_{2\leq j\leq n-1} \|\dd\rho(E_{1j})\beta_p\|^2-\sum_{2\leq i \leq n-1} \|\dd\rho(E_{in})\beta_p\|^2\\
&=(n-p-1)-(p-1)\\
&=n-2p.
\end{aligned}
\end{equation*}
A straightforward computation using the definition of $\owedge$ and Example~\ref{ex:gpowers} gives
\begin{equation}
\langle (R_\mathcal L\owedge \g^{p-2})\beta_p,\beta_p\rangle=(p-1)!
\end{equation}
Thus, we conclude that $B'_{p,n}=\tfrac{n-2p}{(p-1)!}$.

Using the curvature operator $R_\mathcal W$ in \eqref{eq:RW}, by \eqref{eq:krho} and \eqref{eq:LeibnizRule}, we have that
\begin{equation}
\begin{aligned}
\langle\mathcal K(R_{\mathcal W},\wedge^p\R^n)\gamma_p,\gamma_p\rangle &=-\|\dd\rho(E_{13})\gamma_p\|^2-\|\dd\rho(E_{24})\gamma_p\|^2\\
&\quad +2\langle \dd\rho(E_{13})\gamma_p, \dd\rho(E_{24})\gamma_p\rangle\\
&=-8.
\end{aligned}
\end{equation}
On the other hand, using the definition of $\owedge$ and Example~\ref{ex:gpowers}, we have that
\begin{equation*}
\langle (R_{\mathcal W}\owedge\g^{p-2} )\gamma_p,\gamma_p\rangle=4(p-2)!.
\end{equation*}
Thus, we conclude that $C'_{p,n}=\frac{-2}{(p-2)!}$

Finally, consider the irreducible factor $\wedge^4\R^n$. It can be checked that
\begin{equation}
R_{\wedge^4}=E_{12}\otimes E_{34}+E_{34}\otimes E_{12} -E_{13}\otimes E_{24}-E_{24}\otimes E_{13}+E_{14}\otimes E_{23}+E_{23}\otimes E_{14}
\end{equation}
is orthogonal to $\mathcal U\oplus\mathcal L\oplus\mathcal W$ and hence belongs to $\wedge^4\R^n$. Applying \eqref{eq:krho} and \eqref{eq:LeibnizRule},
\begin{equation*}
\begin{aligned}
\langle\mathcal K(R_{\wedge^4},\wedge^p\R^n)\gamma_p,\gamma_p\rangle &= -2\langle\dd\rho(E_{13})\gamma_p,\dd\rho(E_{24})\gamma_p\rangle + 2\langle\dd\rho(E_{14})\gamma_p,\dd\rho(E_{23})\gamma_p\rangle\\
&=-2\langle (-e_3\wedge e_2 + e_1\wedge e_4)\wedge \delta_p,(-e_1\wedge e_4 + e_3\wedge e_2)\wedge \delta_p\rangle\\
&\quad +2\langle (-e_4\wedge e_2 + e_3\wedge e_1)\wedge \delta_p,(-e_1\wedge e_3 + e_2\wedge e_4)\wedge \delta_p\rangle\\
&=8.
\end{aligned}
\end{equation*}
On the other hand, from the definition of $\owedge$ and Example~\ref{ex:gpowers}, we have
\begin{equation*}
\langle (R_{\wedge^4}\owedge \g^{p-2})\gamma_p,\gamma_p\rangle = 2(p-2)!\,\langle e_1\wedge e_2\wedge \delta_p,\gamma_p\rangle \langle e_3\wedge e_4\wedge\delta_p,\gamma_p\rangle= 2(p-2)!.
\end{equation*}
Therefore, we conclude that $D'_{p,n}=\frac{4}{(p-2)!}$.
\end{proof}

\section{Bochner technique in dimension four}
\label{sec:BochnerDim4}

In this section, we combine the Bochner technique and the \emph{Finsler--Thorpe trick} to prove Theorems~\ref{mainthm:dim4} and \ref{mainthm:dim4rigidity} regarding closed $4$-manifolds with $\sec>0$ and $\sec\geq0$. The first of these tools is explained in Subsection~\ref{subsec:Bochner}, so we proceed to briefly discussing the second, see~\cite{BKM,strongpos,ThorpeJDG,Thorpe72,singerthorpe,Zoltek79} for details.

Recall that the (oriented) Grassmannian $\Gr(\R^n)$ of $2$-planes in $\R^n$ is the quadratic subvariety of the unit sphere in $\wedge^2\R^n$ given by the Pl\"ucker relations $\alpha\wedge\alpha=0$, which characterize decomposable elements $\alpha\in\wedge^2\R^n$.
In this context, the sectional curvature function $\sec_R\colon\Gr(\R^n)\to\R$  of a modified algebraic curvature operator $R\in\Sym^2(\wedge^2\R^n)$ is simply the restriction of the quadratic form associated to $R$:
\begin{equation*}
\sec_R(\sigma)=\langle R(\sigma),\sigma\rangle.
\end{equation*}
It is easy to see that the above is independent of the component of $R$ in the subspace $\wedge^4\R^n\subset\Sym^2(\wedge^2\R^n)$. In particular, if there exists $\omega\in\wedge^4\R^n$ such that the operator $(R+\omega)\in\Sym^2(\wedge^2\R^n)$ is positive-definite, then $\sec_R>0$. Remarkably, the converse statement is true in dimensions $\leq4$, as observed by Thorpe~\cite{ThorpeJDG,Thorpe72},
rediscovering a result that was known to Finsler (see \cite{BKM,strongpos} for details).

\begin{proposition}\label{prop:thorpe}
An oriented $4$-manifold $(M,\g)$ has $\sec>0$, respectively $\sec\geq0$, if and only if there exists a continuous function $f\colon M\to\R$ such that the operator $(R+f\,*)\in\Sym^2(\wedge^2 TM)$ is positive-definite, respectively positive-semidefinite.
\end{proposition}

In the above statement, we are using the fact that, in dimension $4$, the subspace $\R\cong\wedge^4\R^4\subset \Sym^2(\wedge^2\R^4)$ is spanned by the \emph{Hodge star} $*\colon \wedge^2 \R^4\to\wedge^2\R^4$, the unique self-adjoint operator such that for all $\alpha\in\wedge^2\R^4$,
\begin{equation}\label{eq:hodgestar}
\alpha\wedge *\,\alpha = \|\alpha\|^2\vol,
\end{equation}
where $\vol=e_1\wedge e_2\wedge e_3\wedge e_4\in\wedge^4\R^4$ is the volume form of $\R^4$.
Moreover, we are using a routine barycenter argument to globalize to $M$ the pointwise statement from each $T_pM$, see~\cite[Rem.~2.3]{strongpos}.
Manifolds (of any dimension) whose curvature operator $R$ admits a positive-definite or positive-semidefinite modification $(R+\omega)\in\Sym^2(\wedge^2 TM)$, $\omega\in\wedge^4 TM$, were systematically studied in~\cite{strongpos,strongnonneg,moduli-flags}. 

Finally, we need some elementary facts regarding self-duality in dimension $4$, see \cite[Ch.~13]{Besse} or \cite{DonaldsonKronheimer} for details. A closed oriented Riemannian $4$-manifold $(M,\g)$ also has a \emph{Hodge star}, defined as the self-adjoint operator $*\colon \wedge^2 TM\to\wedge^2 TM$ for which \eqref{eq:hodgestar} holds for all $\alpha\in\wedge^2TM$, where $\vol\in\wedge^4 TM$ is its volume form.
Since $*^2=\id$, there is an orthogonal direct sum splitting $\wedge^2 TM=\wedge^2_+ TM\oplus \wedge^2_- TM$, where $\wedge^2_\pm TM$ are rank $3$ vector bundles of \emph{self-dual} and $\emph{anti-self-dual}$ $2$-forms, corresponding to the $+1$ and $-1$ eigenspaces of $*$. A standard fact in Hodge~Theory is that there exists an analogous decomposition of the second cohomology
$H^2(M,\R)=\mathcal{H}^+\oplus\mathcal{H}^-$ 
as the direct sum of the spaces $\mathcal H^\pm$ of harmonic self-dual and harmonic anti-self-dual $2$-forms. Writing $b_\pm(M)=\dim \mathcal H^\pm$, we have that the second Betti number of $M$ is 
$b_2(M)=b_+(M)+b_-(M)$, the signature of $M$ is $\tau(M)=b_+(M)-b_-(M)$, and $M$ has \emph{indefinite intersection form} if and only if $b_+(M)>0$ and $b_-(M)>0$.
For instance, $S^2\times S^2$ and $\C P^2\#\overline{\C P^2}$ have indefinite intersection form, as $b_\pm(S^2\times S^2)=b_\pm(\C P^2\#\overline{\C P^2})=1$; while $S^4$ and $\C P^2$ have definite intersection form, as $b_+(S^4)=b_-(S^4)=0$, $b_+(\C P^2)=1$ and $b_-(\C P^2)=0$.

We now turn to the proof of Theorem~\ref{mainthm:dim4}, using the above  facts.

\begin{proof}[Proof of Theorem~\ref{mainthm:dim4}]
According to Proposition~\ref{prop:thorpe}, there exists a continuous function $f\colon M\to\R$ such that $(R+f\,*)\colon\wedge^2 TM\to\wedge^2 TM$ is positive-definite. We claim that $f$ has a zero. 
Since $b_\pm(M)>0$, there exist nonzero harmonic self-dual and anti-self-dual $2$-forms $\alpha_\pm\in\wedge^2_\pm TM$. In particular, by Theorem~\ref{mainthm:KRWedgeSymKN},
\begin{equation}\label{eq:alphaSD}
\langle \mathcal K(\, *,\wedge^2\R^4)\,\alpha_\pm, \alpha_\pm\rangle = 4 \,\langle *\, \alpha_\pm, \alpha_\pm\rangle = \pm 4\, \|\alpha_\pm\|^2.
\end{equation}
Applying the Bochner technique (see Subsection~\ref{subsec:Bochner}), we 
integrate over $M$ the Weitzenb\"ock formula \eqref{eq:weitzenbock} corresponding to the representation $\wedge^2\R^4$, obtaining
\begin{equation}\label{eq:BochnerTech}
\begin{aligned}
0 &= \int_M \langle \Delta \alpha,\alpha\rangle \vol \\
&= \int_M \|\nabla\alpha\|^2 +2 \langle \mathcal K(R,\wedge^2\R^4)\alpha, \alpha\rangle \,\vol \\
&= \int_M \|\nabla\alpha\|^2 +2 \langle \mathcal K(R+f \, *,\wedge^2\R^4)\alpha, \alpha\rangle
-2 \langle \mathcal K(f \, *,\wedge^2\R^4)\alpha, \alpha\rangle  \,\vol \\
&= \int_M \|\nabla\alpha\|^2 +2 \langle \mathcal K(R+f \, *,\wedge^2\R^4)\alpha, \alpha\rangle
\mp 8 f \|\alpha\|^2  \,\vol,
\end{aligned}
\end{equation}
for $\alpha=\alpha_\pm\in\wedge^2_\pm TM$, where the last equality follows from \eqref{eq:alphaSD}.
Since the operator $\mathcal K(R+f \, *,\wedge^2\R^4)$ is positive-definite, $\mp f>0$ would imply that $\alpha_\pm$ vanishes identically, a contradiction. Thus, $f$ has a zero.

The curvature operator $R$ of $(M,\g)$ is positive-definite along $f^{-1}(0)\subset M$, so the statement of the Theorem follows from the existence of $p_\pm\in M$ such that $f(p_-)<0<f(p_+)$ and $R_{p}\colon \wedge^2 T_{p} M\to\wedge^2 T_{p}M$ is not positive-definite for all $p$ in an open neighborhood of $p_\pm\in M$.

Suppose no such $p_+\in M$ exists, so that the curvature operator $R$ is positive-semidefinite at all $p\in M$ such that $f(p)>0$. Thus, the function $f_0=\min\{0,f\}\leq0$ is such that $R+f_0\,*$ is positive-semidefinite. Moreover, since $f$ has a zero, $R+f_0\,*$ is positive-definite at some point, hence on an open set. Setting $\alpha=\alpha_+$ in \eqref{eq:BochnerTech}, it follows that $\alpha_+$ vanishes on this open set, and hence globally on $M$, a contradiction. Thus, there exists $p_+\in M$ with $f(p_+)>0$ and $R_{p}$ not positive-definite for all $p$ in a neighborhood of $p_+$. The existence of $p_-\in M$ is completely analogous.
\end{proof}

\begin{remark}
A consequence of Theorem~\ref{mainthm:dim4} is that if a closed $4$-manifold with indefinite intersection form, such as $S^2\times S^2$ or $\C P^2\#\overline{\C P^2}$, admits a metric with $\sec>0$, then there exists a (nonempty) open subset in $M$ where the curvature operator is positive-definite. It is worth stressing that this conclusion is much weaker than Theorem~\ref{mainthm:dim4}, and indeed can always be achieved up to a small deformation. Namely, given any manifold $(M,\g)$ with $\sec\geq k$ and a neighborhood $U$ of $p\in M$, for all $\varepsilon>0$, there exists a Riemannian metric $\g_\varepsilon$ on $M$ with $\sec_{\g_\varepsilon}\geq k-\varepsilon$, that agrees with $\g$ on $M\setminus U$, and has constant curvature $k$ in a smaller neighborhood of $p\in U\subset M$, see Spindeler~\cite[Cor.~1.6]{spindeler}. On the other hand, Theorem~\ref{mainthm:dim4} ensures that \emph{any} metric with $\sec>0$ on $4$-manifolds as above has positive-definite curvature operator on a subset whose complement has at least two connected components.
\end{remark}

As explained in Subsection~\ref{subsec:Bochner}, the Bochner technique often yields \emph{rigidity} results under curvature assumptions that are not strict. This is precisely the case when relaxing the condition $\sec>0$ to $\sec\geq0$ in the above proof, leading to Theorem~\ref{mainthm:dim4rigidity}.

\begin{proof}[Proof of Theorem~\ref{mainthm:dim4rigidity}]
According to Proposition~\ref{prop:thorpe}, there exists a continuous function $f\colon M\to\R$ such that $R+f \, *$ is a positive-semidefinite operator. If the analogous situation to that of Theorem~\ref{mainthm:dim4} holds, that is, there exist $p_\pm\in M$ such that $f(p_-)<0<f(p_+)$ and $R_{p_\pm}$ is not positive-semidefinite, then (i) holds.

Thus, assume that either for all $p\in M$ such that $f(p)<0$ then $R_p$ is positive-semidefinite, or for all $p\in M$ such that $f(p)>0$ then $R_p$ is positive-semidefinite. Since both cases are analogous, suppose the latter holds. Setting $f_0:=\min\{0,f\}$, this implies that $R+f_0\, *$ is positive-semidefinite. Since $b_+(M)>0$, there exists a nonzero harmonic self-dual $2$-form $\alpha_+\in\wedge^2_+TM$. Then, by \eqref{eq:BochnerTech}, it follows that $\alpha_+$ is parallel and hence $\|\alpha_+\|$ is a positive constant. Using \eqref{eq:BochnerTech} once more, we conclude that $f_0$ vanishes identically, so $R\colon \wedge^2 TM\to\wedge^2 TM$ is positive-semidefinite.

Since $b_-(M)>0$, there exists a nonzero harmonic anti-self-dual $2$-form $\alpha_-\in\wedge^2_- TM$, which must be parallel by the Weitzenb\"ock formula. 
Thus, $\|\alpha_-\|$ is a positive constant, so we may assume $\|\alpha_-\|=\|\alpha_+\|$. 
We claim that the parallel $2$-form $\alpha=\alpha_+ + \alpha_-$ is \emph{pointwise decomposable}, that is, for all $p\in M$ there are $v,w\in T_pM$ such that $\alpha_p=v\wedge w$. Indeed,
\begin{equation*}
\begin{aligned}
\alpha\wedge\alpha &= \alpha_+ \wedge \alpha_+ + 2\,\alpha_+\wedge\alpha_- +\alpha_-\wedge\alpha_- \\
&=\alpha_+ \wedge * \,\alpha_+ - 2\,\alpha_+\wedge *\, \alpha_- -\alpha_-\wedge *\, \alpha_- \\
&= \big(\|\alpha_+\|^2-2\langle \alpha_+,\alpha_- \rangle-\|\alpha_-\|^2\big)\,\vol\\
&= 0,
\end{aligned}
\end{equation*}
so the Pl\"ucker relations are satisfied. Thus, $\alpha$ is a parallel field of $2$-planes on the simply-connected manifold $M$, which hence splits isometrically as a product $(M,\g)=\big(M_1\times M_2,\g_1\oplus\g_2\big)$ of surfaces $(M_1,\g_1)$ and $(M_2,\g_2)$ by the de Rham Splitting Theorem. Since $\sec\geq0$ and $M$ is simply-connected, it follows that $(M_i,\g_i)$ are isometric to $(S^2,\g_i)$ where $\g_i$ have $\sec\geq0$, that is, (ii) holds.
\end{proof}

\end{document}